\documentclass[12pt]{amsart}
\allowbreak
\allowdisplaybreaks
\usepackage[utf8]{inputenc}
\usepackage{amsmath,amsthm}
\usepackage{comment}
\usepackage{amssymb}
\usepackage{stmaryrd}
\usepackage{enumitem}
\usepackage{hyperref}
\usepackage{color}
\usepackage[a4paper,left=3cm,right=3cm,top=3cm,bottom=4.5cm]{geometry}
\usepackage{tikz-cd}
\usepackage{url}
\usepackage[    
backend=biber,
style=alphabetic,
citestyle=alphabetic
]{biblatex}
\usepackage{xpatch}
\usepackage{lineno}
\xpatchcmd{\paragraph}{\normalfont}{{\normalfont\bfseries}}{}{}

\newcommand{\defterm}[1]{\emph{#1}}

\theoremstyle{plain}
	\newtheorem{thm}{Theorem}
	\numberwithin{thm}{section}
	\newtheorem*{thm*}{Theorem}
	\newtheorem{cor}[thm]{Corollary}
	\newtheorem*{cor*}{Corollary}
	\newtheorem{prop}[thm]{Proposition}
	\newtheorem*{prop*}{Proposition}
	\newtheorem{fact}[thm]{Fact}
	\newtheorem*{fact*}{Fact}
	\newtheorem{lem}[thm]{Lemma}
	\newtheorem*{lem*}{Lemma}
	
	\newtheorem*{ex*}{Exercise}
	
	\newtheorem*{claim*}{Claim}
	
	\newtheorem*{conj*}{Conjecture}
	\newtheorem{question}[thm]{Question}
	\newtheorem*{question*}{Question}
	\newtheorem{notation}[thm]{Notation}
	\newtheorem*{notation*}{Notation}
\theoremstyle{definition}
	\newtheorem{Def}[thm]{Definition}
	\newtheorem*{Def*}{Definition}
	\newtheorem{rmk}[thm]{Remark}
	\newtheorem*{rmk*}{Remark}
	
	\newtheorem{soln*}{Solution}
	
	\newtheorem*{note*}{Note}
	\newtheorem{eg}[thm]{Example}
	\newtheorem*{eg*}{Example}	
	\newtheorem{construction}[thm]{Construction}
	\newtheorem*{construction*}{Construction}
	
	\newtheorem*{warning*}{Warning}
	\newtheorem{obs}[thm]{Observation}
	\newtheorem*{obs*}{Observation}	
	\newtheorem{recall}[thm]{Recollection}
	\newtheorem*{recall*}{Recollection}	

\newcommand{\reals}{\mathbb{R}}

\newcommand{\nats}{\mathbb{N}}

\newcommand{\op}{\mathrm{op}}
\newcommand{\id}{\mathrm{id}}

\newcommand{\Hom}{\mathrm{Hom}}

\newcommand{\Ob}{\operatorname{Ob}}

\newcommand{\inv}{^{-1}}
\newcommand{\Aut}{\mathrm{Aut}}
\newcommand{\const}{\mathrm{const}}

\newcommand{\CoCone}{\mathrm{CoCone}}

\newcommand{\Fun}{\mathrm{Fun}}

\newcommand{\Set}{\mathsf{Set}}

\newcommand{\Gpd}{\mathsf{Gpd}}
\newcommand{\Cat}{\mathsf{Cat}}
\newcommand{\Top}{\mathsf{Top}}
\newcommand{\Grp}{\mathsf{Grp}}

\DeclareMathOperator{\Ind}{Ind}
\newcommand{\bbB}{\mathbb{B}}
\newcommand{\sk}{\mathrm{sk}}

\newcommand{\calA}{\mathcal{A}}
\newcommand{\calB}{\mathcal{B}}
\newcommand{\calC}{\mathcal{C}}
\newcommand{\calD}{\mathcal{D}}
\newcommand{\calE}{\mathcal{E}}

\newcommand{\calG}{\mathcal{G}}
\newcommand{\calJ}{\mathcal{J}}

\newcommand{\nerve}{\mathbf{N}}
\newcommand{\sSet}{\mathsf{sSet}}

\newcommand{\Mod}{\mathrm{Mod}}
\newcommand{\Gal}{\mathrm{Gal}}
\newcommand{\Lstp}{\mathrm{Lstp}}
\newcommand{\Lst}{\mathrm{Lst}}

\newcommand{\bbU}{\mathbb{U}}

\DeclareMathOperator{\acl}{acl}

\DeclareMathOperator{\tp}{tp}

\newcommand{\inj}{\mathrm{inj}}

\newcommand{\Tor}{\text{-}\mathrm{Tor}}

\title{Classifying spaces and the Lascar group}
\author{Tim Campion}
\address{Department of Mathematics, Johns Hopkins University, Baltimore MD, United States}
\email{tcampio1@jh.edu}
\author{Greg Cousins}
\address{Department of Mathematics \& Statistics, McMaster University, Hamilton ON, Canada}
\email{cousingd@mcmaster.ca}
\author{Jinhe Ye}
\address{Institut de Math\'ematiques de Jussieu - Paris Rive Gauche, Sorbonne Universit\'e, France}
\email{jinhe.ye@imj-prg.fr}

\begin{document}
\interfootnotelinepenalty=10000
\maketitle
\begin{abstract}
  We show that the Lascar group $\operatorname{Gal}_L(T)$ of a first-order theory $T$ is naturally isomorphic to the fundamental group $\pi_1(|\Mod(T)|)$ of the classifying space of the category of models of $T$ and elementary embeddings. We use this identification to compute the Lascar groups of several example theories via homotopy-theoretic methods, and in fact completely characterize the homotopy type of $|\Mod(T)|$ for these theories $T$. It turns out that in each of these cases, $|\operatorname{Mod}(T)|$ is \emph{aspherical}, i.e. its higher homotopy groups vanish. This raises the question of which homotopy types are of the form $|\Mod(T)|$ in general. As a preliminary step towards answering this question, we show that every homotopy type is of the form $|\mathcal{C}|$ where $\mathcal{C}$ is an Abstract Elementary Class with amalgamation for $\kappa$-small objects, where $\kappa$ may be taken arbitrarily large. This result is improved in another paper.
\end{abstract}

\tableofcontents

\section{Introduction}

\paragraph{The Lascar group.} In \cite{lascar}, Lascar introduced a notion of a Galois group of a complete first-order theory $T$, now known as the Lascar group $\Gal_L(T)$. His main result was a reconstruction theorem: he showed that if the theory $T$ is $\omega$-categorical and the Lascar group over finitely many parameters is always trivial, then one can recover the category of definable sets from the category $\Mod(T)$ of models of the theory. The Lascar group, along with other related notions, has become an important piece of machinery in model theory, see for example \cite{hrushovski2019definability} for some recent progress in understanding the Lascar group. In this paper, we revisit the relationship between $\Mod(T)$ and $\Gal_L(T)$ from the perspective of categorical homotopy theory.

Let us first recall the standard definition of the group $\Gal_L(T)$ in an \emph{ad hoc} categorical framework:

\begin{Def}\label{def:lascar-start}
Given an arbitrary category $\calC$, a full subcategory $\calC_0 \subseteq \calC$, and an object $U \in \calC$, we make the following definitions:
\begin{itemize}
     \item Let $\Lst(\calC,\calC_0,U) \subseteq \Aut_\calC(U)$ denote the subgroup generated by those automorphisms $\alpha \in \Aut_\calC(U)$ such that there exists an $M \in \calC_0$ and a morphism $f: M \to U$ in $\calC$ which is \defterm{fixed by $\alpha$} in the sense that $\alpha  f = f$.
     \item Let $\Gal_L(\calC,\calC_0,U) = \Aut_\calC(U) / \Lst(\calC,\calC_0,U)$.
\end{itemize}
Then for a complete first-order theory $T$, by definition we have
\[\Gal_L(T) := \Gal_L(\Mod(T),\allowbreak \Mod_\kappa(T),\allowbreak \bbU)\]
Here $\Mod(T)$ is the category of models of $T$ with elementary embeddings for morphisms, $\Mod_\kappa(T) \subset \Mod(T)$ is the full subcategory of $\kappa$-small models for some regular cardinal $\kappa>|T|$, and $\bbU$ is a $\kappa$-saturated and strongly-$\kappa$-homogeneous model. It is a theorem that this definition of $\Gal_L(T)$ is independent of the choice of $\kappa$ and of $\bbU$ up to isomorphism.
\end{Def}

Definition \ref{def:lascar-start} positions the Lascar group in the context of pure category theory, but in a way which suffers a number of deficiencies. In particular, the group $\Gal_L(\calC,\calC_0,U)$ depends not only on the category $\calC$, but also on the auxiliary data of $\calC_0 \subseteq \calC$ and $U \in \calC$. Yet in the case of interest, the dependence of $\Gal_L(T)$ on these choices is trivial. One would prefer a description making this independence manifest.

\paragraph{A new perspective.} To shed some light on this phenomenon, let us follow a chain of loose analogies. An analogy between the Lascar group $\Gal_L(T)$ and the absolute Galois group of a field $k$ would liken the choice of $\bbU$ to the choice of an algebraic closure $\bar k$ of $k$ and an embedding $k \to \bar k$. Following another well-known analogy between Galois groups and fundamental groups, this in turn is analogous to the choice of a universal cover $\tilde X$ of a connected space $X$ and a covering map $\tilde X \to X$, or equivalently to a choice of base-point of $X$.

In this paper, we make the ``composite" of these two analogies, relating $\Gal_L(T)$ to the fundamental group of a space, entirely precise. We show (Theorem \ref{thm:mainthm}) that for every first-order theory $T$, there is a space canonically associated to $T$, which we denote $|\Mod(T)|$,\footnote{ See Section \ref{sec:set-con} about set theoretical conventions.} such that $\pi_1(|\Mod(T)|) \cong \Gal_L(T)$. Here $\pi_1(X)$ denotes the fundamental group of a topological space $X$. If $T$ is complete, then $|\Mod(T)|$ is connected, and the base-point-independence of $\pi_1(|\Mod(T)|)$ formally implies the independence of $\Gal_L(T)$ from the choice of a saturated model $\bbU$ (Corollary \ref{cor:lascar-bounded}). Moreover, the space $|\Mod(T)|$ is constructed in an entirely standard way. Namely, $|\Mod(T)|$ is defined to be the classifying space (Definition \ref{def:classifying-space}) of the category $\Mod(T)$ of models of $T$ and elementary embeddings. This is a large category, but we deduce from standard category-theoretic considerations that the classifying space is homotopy equivalent to a small subspace (Proposition \ref{prop:smallequiv}).

\paragraph{Applications.} We develop a few applications of this perspective:
\begin{enumerate}
    \item In Section \ref{sec:lascfund}, we deduce an alternative proof of the invariance of the Lascar group from the choice of $\bbU$, see Corollary \ref{cor:lascar-bounded}. In particular, we obtain weak conditions on a model $\bbU$ such that $\Gal_L(T,\bbU) \cong \Gal_L(T)$, recovering a theorem of \cite{resplendent} (for the notation $\Gal_L(T,\bbU)$, see Definition \ref{def:lascar-group}). We also provide an alternate proof of the usual cardinality bound $|\Gal_L(T)| \leq 2^{|T|}$ (Corollary \ref{cor:lascar-size}).
    \item In Section \ref{sec:examples}, we use the results of the previous section to compute the Lascar groups of several familiar theories using homotopy-theoretic techniques.
    \item The space $|\Mod(T)|$, up to homotopy equivalence, is an invariant of the theory $T$. Since a classifying space is always a CW complex, for the paper, we use the term homotopy type to mean homotopy type of a CW complex. One may ask whether every homotopy type is realized as $|\Mod(T)|$ for a first-order theory $T$. Though we do not know the answer to this question, Section \ref{sec:higher-hty} presents evidence for both an affirmative and negative answer:
    \begin{itemize}
        \item On the one hand, the known ``purely categorical" properties shared by all elementary classes $\Mod(T)$ are nicely summarized in the statement that $\Mod(T)$ is an Abstract Elementary Class (AEC) with amalgamation. We construct, (Theorem \ref{thm:aechty}) for each CW complex $X$ and regular cardinal $\kappa$, an AEC $\calC$ with amalgamation for $\kappa$-presentable objects such that $|\calC| \simeq X$. In a parallel work \cite{campion-ye}, we will provide an alternate construction which realizes the homotopy type of $X$ as $|\calC|$ where $\calC$ is an AEC with amalgamation for \emph{all} objects and no maximal models.
        \item On the other hand, in all the examples of finitary first-order theories $T$ for which we have characterized the homotopy type of $|\Mod(T)|$ in Section \ref{sec:examples}, the space $|\Mod(T)|$ satisfies the restrictive condition of being \emph{aspherical}, meaning its higher homotopy groups vanish.
    \end{itemize}
\end{enumerate}

\paragraph{Questions for Future Work.}
Given the fact that one can identify $\Gal_L(T)$ canonically as $\pi_1(|\Mod(T)|)$, it would be interesting to see if the original work of Lascar in \cite{lascar} can be restated in homotopy-theoretic language.

The Lascar group carries a natural topology, but the fundamental group of a space is merely a discrete group. Thus our results say nothing about the topology on $\Gal_L(T)$. It may be possible to recover this topology in various ways, for example, by enriching $\Mod(T)$ with additional structures such as the ones in \cite{hrushovski2019definability}.

This paper does not consider the relationship between the Lascar group and related notions such as the Kim-Pillay galois group or the Shelah galois group of a theory $T$. We do not know whether these groups admit homotopy-theoretic descriptions. It would be interesting to see what kind of data is needed to enrich $\Mod(T)$ in order to give a homotopy-theoretic definition of the above groups.

Besides fundamental groups, one may also consider other homotopy invariants of the space $|\Mod(T)|$ such as higher homotopy groups, homology groups, and cohomology groups. We do not consider such invariants in this paper because in the theories $T$ which we consider, the higher homotopy groups of $|\Mod(T)|$ vanish. Consequently, in these cases, the space $|\Mod(T)|$ is homotopy equivalent to the classifying space of the group $\pi_1(|\Mod(T)|) \cong \Gal_L(T)$ considered as a discrete group (cf. Example \ref{eg:kgn}), and therefore the homology and cohomology groups of $|\Mod(T)|$ are simply the group homology and cohomology of $\Gal_L(T)$ considered as a discrete group.

In \cite{GKK}, a notion of homology of types was defined and in \cite{DKL}, it was shown that the first homology group of a given type $p$ is given by the abelianization of the relativized Lascar groups. The relationship between their construction and the one considered here is unclear. 

The model theoretic significance of the higher homotopy groups of $|\Mod(T)|$ -- if indeed they can be nonvanishing -- is not clear to us. It has been suggested that non-trivial elements in $\pi_n(|\Mod(T)|)$ might be related to failure of $n+1$-almagation of models of $T$. However, $n+1$-amalgamation involves finding cocones on certain diagrams with contractible classifying space, whereas the homotopy group $\pi_n(|\Mod(T)|)$ has to do with diagrams with classifying space homotopy equivalent to the $n$-sphere $S^n$. Thus, as we will see in Example \ref{eg:tri-free-ran}, such a connection remains unclear.

In \cite{campion-ye}, the first and third authors improve the results of Section \ref{sec:higher-hty}, showing that for every small CW complex $X$, there exists an AEC $\calC$ with amalgamation for all objects and no maximal models, such that $|\calC| \simeq X$. In fact, $\calC$ consists of the models of certain theory in $L_{\kappa^+,\omega}$ where $\kappa = \max(\aleph_0, \# X)$ (where $\# X$ is the number of cells of $X$), with strong embeddings for morphisms.

\paragraph{Overview.}
We begin in Sections \ref{sec:cathty-bkgd} and \ref{sec:mod-bkgd} by recalling some background material in category, homotopy theory and model theory, respectively. In Section \ref{sec:lascfund}, the heart of the paper, we prove our main result (Theorem \ref{thm:mainthm}), exhibiting the Lascar group as the fundamental group of a space, and deduce that the Lascar group is bounded (Corollary \ref{cor:lascar-bounded}), and give a bound on its cardinality (Corollary \ref{cor:lascar-size}). In Section \ref{sec:examples}, we give some toy applications, computing the Lascar groups of several familiar theories with methods of a homotopy-theoretic and category-theoretic flavor. In fact, in all of the examples we consider, we are able to go further and determine the complete homotopy type of the space $|\Mod(T)|$, and find that in these examples $|\Mod(T)|$ is aspherical (i.e. its higher homotopy groups vanish). In Section \ref{subsec:higher-hty-1} we consider the question of which homotopy types can be realized in the form $|\Mod(T)|$. In this paper, we attain only partial results, including Observation \ref{obs:nmm}, \ref{obs:ap} and Theorem \ref{thm:aechty}. We conclude in Section \ref{subsec:higher-hty-2} with a discussion of known categorical criteria on a category $\calC$ implying that $|\calC|$ is aspherical, and find that none of the criteria we consider are satisfied by $\Mod(T)$ for a general finitary first-order theory $T$.

\subsection{Set-theoretic conventions.}\label{sec:set-con}
We fix two strongly inaccessible cardinals $\lambda_1 < \lambda_2$. For each ordinal $\alpha$, let $V_\alpha$ denote the sets of rank less than $\alpha$. Sets in $V_{\lambda_1}$ are called \defterm{small}, sets in $V_{\lambda_2}$ are called \defterm{moderate}, any other set is called \defterm{large}. Models (including ``monster models"), by convention, are assumed to be small. Topological spaces are assumed to be moderate. Categories have no cardinality restrictions. A category is called \defterm{small} if its sets of objects and morphisms are small, and a category with small homsets is called \defterm{locally small}. 

Set-theoretical considerations enter this paper merely as a convenience for two reasons. One is that we would like to be able to talk about the (large) category of all (moderate) categories. The other is that we would like to consider the classifying space of a moderate category.  Both uses can be circumvented, the first by standard circumlocutions and the second by Proposition \ref{prop:smallequiv}, which shows that the topological spaces we are interested in are homotopy equivalent to canonical small subspaces. Thus our results will continue to hold, \emph{mutatis mutandis}, in ZFC.

In particular, we do \emph{not} use ``monster models" which are saturated in their own cardinality.

\subsection{Notational conventions.}\label{not-con}
In this paper, composition in a category is written as follows. If $A \overset f \to B \overset g \to C$ is a composable pair in a category $\calC$, then the composite is written in the usual way as $A \overset{gf} \to C$. This convention includes the case where $\calC$ is a groupoid.

We refer to several specific categories in the paper. $\Set$ is the category of small sets and functions. $\Cat$ is the category of moderate categories. $\Top$ is the category of moderate topological spaces. $\Gpd$ is the full subcategory of $\Cat$ consisting of (moderate) groupoids (see Section \ref{sec:cat-bkgd}), and $\Grp$ is the category of (small) groups. $\sSet$ is the category of moderate simplicial sets. For a first-order theory $T$, $\Mod(T)$ is the category of (small) models of $T$ and elementary embeddings. This differs from the convention of \cite{lpacc} for example, where $\Mod(T)$ is used for the category of models of $T$ with \emph{homomorphisms} as morphisms, and $\operatorname{Elem}(T)$ is used for our $\Mod(T)$.

For a category $\calC$, we will typically use the notation $\Hom_\calC(-,-)$ for the $\Hom$ sets, and we will drop the subscript $\calC$ when the underlying category $\calC$ is clear from the context. In some cases we may also write $\calC(X,Y)$ instead of $\Hom_\calC(X,Y)$ to emphasize the category in which the homset is being taken.

\paragraph{Background.}
Category-theoretically, we assume only that the reader is familiar with the concepts of category, functor, natural transformation, subcategories, full and faithful functors, opposite categories, isomorphism, and equivalence of categories. A familiarity with notions of universal properties is also helpful, but not strictly necessary. We will review several other basic and advanced concepts of category theory in Section \ref{sec:cat-bkgd}. We will assume a bit more categorical sophistication in Section \ref{sec:higher-hty}.

Homotopy-theoretically, we assume only that the reader is familiar with the notions of topological space and continuous map (although a more sophisticated reader should feel free to substitute their favorite model of homotopy theory). We will review several basic and advanced concepts of homotopy theory in Section \ref{sec:hty-bkgd}. 
In Section \ref{sec:higher-hty}, we will assume a bit more homotopical sophistication.

Model-theoretically, we assume only that the reader is familiar with the notions of first-order structure, first-order theory, elementary embedding, and completeness. We will review several more advanced concepts of model theory in Section \ref{sec:mod-bkgd}.

\paragraph{Acknowledgements.}
First and foremost, we would like to thank Kyle Gannon for suggesting the study of the space $|\Mod(T)|$, and to Omar Antol\'{i}n-Camarena, whose note \cite{oac} suggested, however indirectly, that this might not be an uninteresting thing to do. The genesis of this project was serendipitous and would not have been possible without Dominic Culver, Stephan Stolz, and the Notre Dame Graduate Topology Seminar. We would like to thank Reid Barton, Mark Behrens, Will Boney, Adrian Clough, Mike Haskel, Alex Kruckman, Edoardo Lanari, Mike Lieberman, Zhen Lin Low, Anand Pillay, Chris Schommer-Pries, Sergei Starchenko and Jir{\'i} Rosick{\'y} for helpful discussions. And we would like to express our gratitude to the annoymous referee for numerous useful comments. TC was supported in part by NSF grant DMS-1547292. JY was supported in part by NSF grant DMS-1500671.


\section{Background in category theory and homotopy theory}\label{sec:cathty-bkgd}

\subsection{Background in category theory}\label{sec:cat-bkgd}

\paragraph{Basic notions}
There are many excellent introductions to category theory available, for example the classic \cite{cwm} or the more modern \cite{catctxt}, which is freely available from the author's website. We recall some basic definitions here, but our treatment is necessarily brief.

\begin{recall}[Equivalence of Categories]\label{rec:eqcats}
A functor $F: \calC \to \calD$ is an \defterm{equivalence of categories} if there is a functor $G: \calD \to \calC$ and natural isomorphisms $\id_\calC \cong GF$ and $\id_\calD \cong FG$. In this case we refer to $G$ as a \defterm{weak inverse} to $F$. A basic theorem says that a functor $F$ is an equivalence of categories if and only if it is full, faithful, and essentially surjective (i.e. surjective on isomorphism classes of objects). Moreover, in this case a weak inverse may be constructed by arbitrarily fixing, for each $D \in \calD$, an object $G(D) \in \calC$ such that $F(G(D)) \cong D$, and an isomorphism $\iota_D: F(G(D)) \to D$. Then $G$ is defined on morphisms $f: D \to D'$ by $G(f) = F\inv(\iota_{D'}\inv f \iota_D)$. We use $\calC\simeq \calD$ to denote that $\calC$ and $\calD$ are equivalent.
\end{recall}

\begin{recall}[Skeleton]\label{rec:skeleton}
A \defterm{skeletal} category is a category where any two isomorphic objects are equal. Any equivalence of categories between skeletal categories is an isomorphism of categories. A \defterm{skeleton} of a category $\calC$ is an equivalent skeletal category. Every category has a skeleton, unique up to isomorphism of categories, which may be constructed by choosing an arbitrary representative of each isomorphism class of objects of $\calC$, and taking $\sk \calC \subseteq \calC$ to be the full subcategory on these objects. The inclusion $\sk \calC \to \calC$ is unique up to non-unique natural isomorphism. Two categories are equivalent if and only if their skeleta are isomorphic.
\end{recall}

\begin{recall}[Size]\label{rec:size}
The \defterm{size} $\#(\calC)$ of a category $\calC$ is the cardinality of its set of morphisms. Note that on account of identity morphisms, this is always at least the number of objects of $\calC$. If $\kappa$ is a cardinal, we say that $\calC$ is \defterm{$\kappa$-small} if $\#(\calC) <\kappa$. A more important notion is the \defterm{essential size} of a category $\calC$, i.e. the size of a skeleton of $\calC$.
\end{recall}

\begin{recall}[Connected]\label{rec:conncat}
A category $\calC$ is \defterm{strongly connected} if its underlying directed graph is strongly connected, i.e. for every two objects $x,y \in \calC$, there is a morphism from $x$ to $y$. A category $\calC$ is \defterm{connected} if its underlying directed graph is weakly connected, i.e. for every two objects $x,y \in \calC$ there exists a zigzag of morphisms $x=x_0 \leftarrow x_1 \to x_2 \leftarrow \dots \leftarrow x_{n-1} \to x_n = y$.
\end{recall}

\begin{recall}[Adjunction]
An \defterm{adjunction} consists of a pair of functors $F: \calC {}^\to_\leftarrow \calD: U$ equipped with a family of bijections $\Hom_\calD(FC,D) \cong \Hom_\calC(C,UD)$, natural in $C,D$. In this situation $F$ is said to be \defterm{left adjoint} to $U$, and $U$ is \defterm{right adjoint} to $F$. Alternatively, an adjunction may be specified by natural transformations $\eta: \id_\calC \Rightarrow UF$ (the \defterm{unit}) and $\varepsilon: FU \Rightarrow \id_\calD$ (the \defterm{counit}) satisfying certain equations. The bijection of homsets encodes a \emph{universal property}: in order to define a morphism $FC \to D$, it suffices to define a morphism $C \to UD$.
\end{recall}

\begin{recall}[Monomorphism]\label{rec:mono}
Let $\calC$ be a category. A \defterm{monomorphism} in $\calC$ is a morphism $f : D \to E$ such that for every two morphisms $g,h: C \to D$, if $fg=fh$, then $g=h$. For example, in the category $\Set$ of sets and functions, the monomorphisms are precisely the injective functions.
\end{recall}

\begin{recall}[Fibration, Slice Category]\label{rec:fibration}
The notion of a \emph{Grothendieck opfibration} is used as a convenience in Section \ref{sec:examples}. Let $F: \calE \to \calB$ be a functor. Let $f: B \to B'$ be a morphism in $\calB$ and $E \in \calE$ an object with $F(E) = B$. A \defterm{cocartesian lift} of $f$ to $E$ is a morphism $g: E \to E'$ with $F(g) = f$ satisfying the following universal property. For every $f': B' \to B''$ and every $g'' : E \to E''$ with $F(g'') = f'f$, there exists a unique $g': E' \to E''$ such that $F(g') = f'$ and $g'g = g''$. By the universal property, cocartesian lifts are unique up to unique vertical isomorphism; that is, if $g,: E \to E'$ and $\bar g: E \to \bar E$ are two cocartesian lifts of $f$ to $E$, then there is a unique map $\iota: E' \to \bar E$ such that $F(\iota) = \id_{B'}$ and $\iota g = \bar g$; moreover, $\iota$ is an isomorphism. The functor $F$ is said to be a \defterm{Grothendieck opfibration} if for every $f: B \to B'$ and $E \in \calE$ such that $F(E) = B$, there is a cocartesian lift of $f$ to $E$. Note that any isomorphism in $\calE$ is cocartesian, and conversely if $f: B \to B'$ is an isomorphism, then any cocartesian lift of $f$ is an isomorphism.

Let $F: \calC \to \calD$ be a functor and $D \in \calD$. The \defterm{fiber} $F\inv(D)$ of $F$ over $D$ is the subcategory of $\calC$ consisting of objects $C$ with $F(C)=D$ and morphisms $f$ with $F(f) = \id_D$. The \defterm{slice category}  $F \downarrow D$ of $F$ over $D$ is the following category. The objects are pairs $(C, f)$ where $C \in \calC$ and $f: F(C) \to D$ is a morphism in $\calD$. The morphisms from $(C,f)$ to $(C',f')$ are morphisms $g: C \to C'$ such that $f' F(g)  = f$. If the functor $F$ is clear from context, we may write $\calC \downarrow D$ instead of $F \downarrow D$. If $h: D \to D'$ is a morphism in $\calD$, then there is a \defterm{reindexing functor} $h_\ast: F \downarrow D \to F \downarrow D'$ defined by $(C,f) \mapsto (C,hf)$. There is an inclusion $F\inv(D) \to F \downarrow D$, and if $F$ is a Grothendieck opfibration the inclusion functor has a left adjoint denoted $(C,f) \mapsto f_\ast(C)$, defined by taking $f_\ast(C)$ to be the codomain of an arbitrary cocartesian lift of $f$ to $C$. If $F$ is a Grothendieck opfibration, then there is also a \defterm{reindexing} functor $h_\ast: F\inv(D) \to F\inv(D')$ defined similarly. The two types of reindexing functor are compatible in that there is an isomorphism $(hf)_\ast(C) \cong h_\ast(f_\ast(C))$, natural in $(C,f) \in F \downarrow D$.
\end{recall}

\begin{recall}[(Co)Cones, (Co)Limits]
An \defterm{initial object} in a category $\calC$ is an object $0 \in \calC$ with the following univeral property: every object $C \in \calC$ admits a unique morphism $0 \to C$. If $\calC$ has an initial object, then it is unique up to unique isomorphism, so we speak of \emph{the} initial object, and denote it $0$.

A \defterm{diagram} in a category $\calC$ consists of a category $J$ (the \defterm{shape} of the diagram) and a functor $X: J \to \calC$. When thinking of a functor as a diagram, we will often denote its application to objects and morphisms using subscript notation, so that if $j \in J$ we write $X_j = X(j)$ and if $u: j \to j'$ is a morphism in $J$, we write $X_u: X_j \to X_{j'}$. We will sometimes also use notation such as $(X_j)_{j \in J}$ to denote a functor $X$ of shape $J$, or even $(X_j)_j$ when $J$ is implicit. A \defterm{cocone} on a diagram $X: J \to \calC$ consists of an object $V \in \calC$ (the \defterm{vertex} of the cone) and morphisms $\lambda_j: X_j \to V$ for all $j \in J$ (the \defterm{legs} of the cone) such that $\lambda_j X_u = \lambda_{j'}$ for all $u: j' \to j$ in $J$. If $(V,\lambda)$ and $(V',\lambda')$ are cocones on $X$, then a \defterm{morphism of cocones} $f: (V,\lambda) \to (V',\lambda')$ consists of a morphism $f: V \to V'$ such that $f\lambda_j = \lambda'_j$ for all $j \in J$; thus there is a category $\CoCone(X)$ of cones on $X$. A \defterm{colimit} of $X$ is an initial object in $\CoCone(X)$; if a colimit of $X$ exists it is unique up to unique isomorphism of cocones on $X$, so we speak of \emph{the} colimit and denote it $(\varinjlim X, \iota^X)$ or $(\varinjlim_{j \in J} X_j, \iota^X)$ or just $(\varinjlim_j X_j,\iota^X)$ when $J$ is implicit. Moreover, we will often leave the colimiting cocone $\iota$ implicit, referring to $\varinjlim X$ abusively as the colimit of $X$.

For any functor $F: \calC \to \calD$ and any category $I$, there is an induced map from $I$-shaped diagrams in $\calC$ to $I$-shaped diagrams in $\calD$ as well as an induced map on cocones over such diagrams. We say that $F$ \defterm{preserves} $I$-shaped colimits if it carries colimiting cocones on $I$-shaped diagrams to colimiting cocones on $I$-shaped diagrams. Any left adjoint functor preserves all colimits.

\defterm{Terminal objects} (denoted $1$), \defterm{cones}, and \defterm{limits} of $F$ (denoted $\varprojlim F$), and preservation thereof, are defined dually.
\end{recall}

\begin{recall}[Filtered]
Perhaps less familiar is the notion of a \defterm{filtered} category, i.e. a category $\calC$ such that every functor $F: \calJ \to \calC$, with $\calJ$ finite, admits a cocone; to show that $\calC$ is filtered, it suffices to check the following 3 cases: (1) $\calJ$ is empty, (2) $\calJ$ consists of two objects with no nonidentity morphisms, (3) $\calJ$ is the category depicted as $0 {}\rightrightarrows 1$. Filtered categories generalize directed posets; in particular, if $\calJ$ is a preorder (i.e. a category where each homset as at most one element), then $\calJ$ is a filtered category if and only if it is a directed preorder. More generally, if $\kappa$ is a regular cardinal, a \defterm{$\kappa$-filtered} category is a category $\calC$ such that every functor $F: \calJ \to \calC$ with $\calJ$ $\kappa$-small (cf. Recollection \ref{rec:size}), admits a cocone.
\end{recall}

\begin{eg}[Examples of (Co)limits]\label{eg:colim}
Let $J$ be a discrete category, i.e. one with no nonidentity morphisms. Then the data of a diagram of shape $J$ in $\calC$ is equivalent to the data a function $X: \Ob J \to \Ob \calC$. A colimit of shape $J$ is called a \defterm{coproduct}, and denoted $\amalg_{j \in J} X_j$, or by $X_1 \amalg X_2$ in the binary case. The category $\Set$ has coproducts, given by disjoint union. The category $\Top$ has coproducts, also given by disjoint union. Incidentally, a set $1$ with one element is a terminal object in $\Set$, and the discrete category $1$ with one object is a terminal object in $\Cat$.

In general, small colimits exist in $\Set$ and in $\Top$. They are computed by $\varinjlim_{j \in J} X_j = \amalg_{j \in J} X_j / \sim$ where $\sim$ is the equivalence relation \emph{generated} by setting $x \in X_j$ equivalent to $X_u(x)$ for every $u: j \to j'$. There are some cases where the equivalence relation so generated admits a simplified description.

For example, let $J$ be the category $1 \leftarrow 0 \to 2$. A diagram of shape $J$ is called a \defterm{span}, and consists of three objects with two nonidentity morphisms between them in the shape $B \leftarrow A \to C$. A colimit of shape $J$ is called a \defterm{pushout}, and denoted $B \cup_A C$. In $\Set$, pushouts are particularly easy to describe when the maps $B \overset i \leftarrow A \xrightarrow j C$ are injective: in this case $B \cup_A C$ is $B \amalg C$ modulo the equivalence relation identifying $i(a)$ with $j(a)$.

For another example, if $J$ is small and filtered and $F: J \to \Set$ is a $J$-indexed diagram in $\Set$, then $\varinjlim_j F(j) = \amalg_j F(j) / \sim$, where $\sim$ identifies $x \in F(j)$ with $x' \in F(j')$ if and only if there is $j'' \in J$ and morphisms $j' \xrightarrow f j'' \overset{f'}{\leftarrow} j'$ such that $F(f)(x) = F(f')(x')$. A similar construction works in $\Cat$, and in many other categories.
\end{eg}

\begin{Def}[JEP, AP]\label{def:JEP-AP}
Let $\calC$ be a category.
\begin{enumerate}
\item We say that $\calC$ has the \defterm{joint embedding property (JEP)} if for any finite, discrete diagram in $\calC$ there is a cocone. In other words, for every $A,B \in \calC$, there is an object $C$ and two morphisms $A\rightarrow C$, $B\rightarrow C$.
\item We say that $\calC$ has the \defterm{amalgamation property (AP)} if for any \defterm{span} in $\calC$, i.e. a diagram of the form $B \overset{f}{\leftarrow} A \xrightarrow f C$, there is a cocone. In other words, there is an object $D$ and morphisms $B \overset{g'}\to D \overset{f'}\leftarrow C$ such that $g'f = f'g$.
\end{enumerate}
\end{Def}

\begin{eg}\label{eg:modt-jep-ap}
$\Mod(T)$ has the joint embedding property and amalgamation property. This follows from compactness and the Downward L\"owenheim-Skolem Theorem.
\end{eg}

\paragraph{Groupoids}

\begin{recall}[Groupoid]\label{rec:gpd}
A \defterm{groupoid} is a category where every morphism is an isomorphism. We denote by $\Gpd$ the category of moderate groupoids. If $G$ is a group, then there is a 1-object groupoid, denoted $\bbB G$, whose morphisms are the elements of $G$, with composition given by multiplication in $G$. Note that if $G$ and $H$ are groups, then $\bbB G$ and $\bbB H$ are equivalent as groupoids if and only if they are isomorphic as groups. 
Moreover, $G\to H$ is an isomorphism of groups iff the induced functor $\bbB G\to \bbB H$ is an equivalence of categories. If $\Gamma$ is any groupoid and $x \in \Gamma$ is an object, then we will denote $\pi_1(\Gamma,x):= \Hom_\Gamma(x,x)$. There is a canonical inclusion functor $\iota_x: \bbB \pi_1(\Gamma,x) \to \Gamma$, which is fully faithful. We say that a groupoid $\Gamma$ is \defterm{trivial} if it is nonempty and for every two objects $x,y \in \Gamma$ there is a unique morphism $x \to y$. Equivalently, $\Gamma$ is trivial if it is equivalent to $\bbB G$ for $G$ the trivial group, i.e. if it is equivalent to the terminal category $1$ (cf. Example \ref{eg:colim}).
\end{recall}

\begin{recall}[Connected Groupoids]\label{rec:conngpd}
Let $\Gamma$ be a connected groupoid (cf. Recollection \ref{rec:conncat}). Then $\Gamma$ is strongly connected, and so all of its objects are isomorphic. For if $x_0 \leftarrow x_1 \to x_2 \leftarrow \dots \leftarrow x_{n-1} \to x_n$ is a zigzag, then by taking the inverses of all the backward-pointing arrows and composing, we obtain an isomorphism $x_0 \to x_n$. It follows that if $\Gamma$ is a connected groupoid and $x \in \Gamma$ is an object, then the inclusion $\bbB \pi_1(\Gamma,x) \to \Gamma$ is full, faithful, and essentially surjective and so an equivalence of groupoids (cf. Recollection \ref{rec:eqcats}). Thus when $\Gamma$ is a nonempty connected groupoid, we may sometimes refer to $\pi_1(\Gamma)$ loosely. In particular, when $\Gamma$ is connected, the groups $\pi_1(\Gamma,x)$ are isomorphic for different $x \in \Gamma$.

When $\Gamma$ is not connected, it is a disjoint union of its connected components, and equivalent to its skeleton, which is a disjoint union of one-object groupoids. Thus it is often fruitful to think of a groupoid simply as a collection of groups.
\end{recall}

\begin{recall}[Fundamental Groupoid of a Category]\label{rec:catfund}
The inclusion $\Gpd \to \Cat$ has a left adjoint denoted $\Pi_1: \Cat \to \Gpd$. The gropoid $\Pi_1(\calC)$ is called the \defterm{fundamental groupoid of $\calC$}, or sometimes the \defterm{free groupoid on $\calC$}. For $\calC \in \Cat$, we will denote by $\llbracket-\rrbracket: \calC \to \Pi_1(\calC)$ the unit of this adjunction, so that for $f: C \to C'$ a morphism in $\calC$, $\llbracket f \rrbracket$ denotes the equivalence class of $f$ considered as a morphism in $\Pi_1(\calC)$. The groupoid $\Pi_1(\calC)$ is constructed by ``freely adjoining inverses to $\calC$"; see \cite[Chapter I.1]{gabriel-zisman} for an explicit description. In practice, it suffices to know that there is a functor $\llbracket - \rrbracket: \calC \to \Pi_1(\calC)$ and that every morphism of $\Pi_1(\calC)$ is invertible. Beware that $\llbracket-\rrbracket: \calC \to \Pi_1(\calC)$ is generally far from faithful, even on automorphisms. For $C \in \calC$, we also write $\pi_1(\calC,C) := \Hom_{\Pi_1(\calC)}(C,C) = \pi_1(\Pi_1(\calC),C)$. If $\calC$ is connected and nonempty, then we may loosely speak of $\pi_1(\calC)$. We note that $\Pi_1(\calC)$ has the same objects as $\calC$, and is connected iff $\calC$ is. Moreover, we have $\#(\Pi_1(\calC)) \leq \aleph_0 + \#(\calC)$, and $\Pi_1(\sk \calC) \simeq \sk \Pi_1(\calC)$.

This adjunction is in fact a 2-categorical adjunction. We will not spell out what this means, but we note a few consequences, which can also be deduced directly from the explicit description in \cite[Chapter I.1]{gabriel-zisman}. First, if $F,G: \calC\rightrightarrows \Gamma$ are two functors to a groupoid and $\alpha: F \Rightarrow G$ is a natural isomorphism, and if $\tilde F, \tilde G: \Pi_1(\calC)\rightrightarrows \Gamma$ are the functors induced by the universal property of the adjunction, then there is an induced natural isomorphism $\tilde \alpha: \tilde F \Rightarrow \tilde G$. It follows that if $\calC$ and $\calD$ are equivalent categories, then their fundamental groupoids are also equivalent.
\end{recall}

\paragraph{Accessible categories}

The theory of accessible categories will be used in Section \ref{sec:higher-hty}; all of Section \ref{sec:lascfund} also generalizes to this setting, but we do not assume familiarity with accessible categories in Section \ref{sec:lascfund}. For the theory of accessible categories, we refer the reader to \cite{lpacc}.

\begin{Def}[Accessible Category]\label{def:acc}
Let $\kappa$ be a small regular cardinal. A \defterm{$\kappa$-accessible category} is a moderate,\footnote{Most treatments of accessible categories such as \cite{lpacc} work in a slightly different foundational setup than we do; the ``sets" of such treatments should be identified with the \emph{small} sets in our treatment.} locally small category with $\kappa$-filtered colimits (i.e. colimits of functors $F: J \to \calC$ where $J$ is small and $\kappa$-filtered), and an essentially small full subcategory $\calC_\kappa \subseteq \calC$ of \defterm{$\kappa$-presentable} objects -- objects $C$ such that $\Hom_\calC(C,-): \calC \to \Set$ commutes with $\kappa$-filtered colimits for $C \in \calC_\kappa$ -- and moreover every object $X \in \calC$ is the colimit of a small, $\kappa$-filtered diagram in $\calC_\kappa$. In this case the $\kappa$-presentable objects are precisely the retracts of objects of $\calC_\kappa$: for if $X$ is $\kappa$-presentable, then we write $X = \varinjlim_i X_i$ with $X_i \in \calC_\kappa$, and by $\kappa$-presentability of $X$ (and the construction of filtered colimits in $\Set$), the identity map $X \to X$ must factor through some stage $X_i$ of the colimit, so that $X$ is a retract of $X_i$. The $\kappa$-presentable objects are closed under $\kappa$-small colimits \cite[Proposition 1.16]{lpacc}.

It may be comforting to know that every instance of ``filtered colimit" in this definition may be replaced with ``directed colimit" because every $\kappa$-filtered category admits a cofinal $\kappa$-directed poset \cite[Remark 1.12]{lpacc}. We say \defterm{finitely accessible} for $\aleph_0$-accessible, and \defterm{accessible} to mean $\kappa$-accessible for some (unspecified) $\kappa$.
\end{Def}

\begin{eg}[$\Mod(T)$]\label{eg:modt-acc}
Let $T$ be a first-order theory. The category $\Mod(T)$ of (small) models of $T$ is accessible with all filtered colimits \cite[Remark 2.3(4)]{beke-rosicky}.\footnote{In \cite{beke-rosicky} a different foundational setup is used. Sets in \cite{beke-rosicky} should be identified with our \emph{small} sets.} More precisely, filtered colimits are given by union \cite[Remark 2.3(4)]{beke-rosicky}, for $\kappa\geq |T|^+$ the $\kappa$-presentable objects $\Mod_\kappa(T)$ are precisely the models of cardinality $<\kappa$ \cite[Remark 2.3(4)]{beke-rosicky}, and so the Downward L\"owenheim-Skolem Theorem implies that $\Mod(T)$ is $\kappa$-accessible for $\kappa \geq |T|^+$. Notably, $\Mod(T)$ satisfies AP and, if $T$ is complete, also JEP (cf. Definition \ref{def:JEP-AP}).
\end{eg}

\begin{recall}[The $\Ind$ Construction]\label{def:ind}

If $\calC$ is a small category and $\kappa < \lambda_1$ is a regular cardinal (recall from Section \ref{sec:set-con} that $\lambda_1$ denotes the size of the universe of small sets), we will denote by $\Ind_\kappa(\calC)$ the following category. An object is a small, $\kappa$-filtered diagram $(C_i)_{i \in I}$ in $\calC$, and a morphism from $(C_i)_{i \in I}$ to $(D_j)_{j \in J}$ is a natural transformation $\varinjlim_i \Hom_\calC(-,C_i) \to \varinjlim_j \Hom_\calC(-,D_j)$, where the colimit is taken in the presheaf category $\Fun(\calC^\op, \Set)$. Composition is as in $\Fun(\calC^\op,\Set)$. Clearly, $\Ind_\kappa(\calC)$ is equivalent to the full subcategory of the functor category $\Fun(\calC^\op,\Set)$ consisting of those functors which are small, $\kappa$-filtered colimits of representable functors. Alternatively, $\Ind_\kappa(\calC)$ may be regarded as simply a category of diagrams with homsets given by the formula $\Hom((C_i)_{i \in I}, (D_j)_{j \in J}) = \varprojlim_{i \in I} \varinjlim_{j \in J} \Hom(C_i, D_j)$. We write $\Ind$ for $\Ind_\omega$.

There is a natural fully faithful functor $y_\calC: \calC \to \Ind_\kappa(\calC)$ sending $C$ to the one-object diagram at $C$ (viewing $\Ind_\kappa(\calC)$ as a category of presheaves, this corresponds to the Yoneda embedding), and we will often identify $\calC$ with its image under this embedding. The category $\Ind_\kappa(\calC)$ has all small, $\kappa$-filtered colimits, computed as in the presheaf category $\Fun(\calC^\op, \Set)$, and the objects of $\calC$ are $\kappa$-presentable in $\Ind_\kappa(\calC)$. In fact, the $\kappa$-presentable objects of $\Ind_\kappa(\calC)$ are precisely the retracts of objects of $\calC$. Moreover, every object of $\Ind_\kappa(\calC)$ is canonically a small, $\kappa$-filtered colimit of objects of $\calC$. Thus $\Ind_\kappa(\calC)$ is a $\kappa$-accessible category. A central theorem of the theory of accessible categories says that a category is $\kappa$-accessible if and only if it is equivalent to $\Ind_\kappa(\calC)$ for some small category $\calC$. The category $\Ind_\kappa(\calC)$ has the universal property that if $G:\calC \to \calD$ is a functor where $\calD$ has small, $\kappa$-filtered colimits, then there is a functor $\tilde G: \Ind_\kappa(\calC) \to \calD$ preserving small, $\kappa$-filtered colimits such that $\tilde G|_\calC = G$; $\tilde G$ is unique up to unique natural isomorphism.
\end{recall}

\begin{recall}[The restricted $\Ind$ construction]\label{def:ind-res}

In Section \ref{subsec:higher-hty-1}, we will use the following variant of the $\Ind_\kappa$ construction, as expanded on in \cite{mo-zl}. For regular cardinals $\kappa \leq \mu < \lambda_1$, we define $\Ind_\kappa^\mu(\calC)$ to be the full subcategory of $\Ind_\kappa(\calC)$ on diagrams which are $\mu$-small. The category $\Ind_\kappa^\mu(\calC)$ inherits many properties from the larger category $\Ind_\kappa(\calC)$; see \cite{mo-zl} and \cite{zhen-lin} for details.
We will only use this construction in the case $\kappa = \omega$; we write $\Ind^\mu(\calC)$ for $\Ind_\omega^\mu(\calC)$.

In particular, $\Ind^\kappa(\calC)$ has $\kappa$-small filtered colimits, computed as in $\Ind(\calC)$, so that if $C \in \calC$, then $\Ind^\kappa(\calC)(C,-)$ commutes with $\kappa$-small filtered colimits and every object of $\Ind^\kappa(\calC)$ is a $\kappa$-small, filtered colimit of objects of $\calC$ \cite[Proposition 1.11, Example 1.8]{zhen-lin}. If $\calC \to \calD$ is a functor where $\calC$ is essentially small and $\calD$ has $\kappa$-small filtered colimits, then we obtain an induced functor $\Ind^\kappa(\calC) \to \calD$ preserving $\kappa$-small filtered colimits. 
\end{recall}

\begin{lem}\label{lem:ind-mono}
Let $\kappa \leq \mu \leq \lambda_1$ be regular cardinals. If $\calC$ is a category of monomorphisms, then $\Ind_\kappa^\mu(\calC)$ is a category of monomorphisms.
\end{lem}
\begin{proof}
Let $f: Y \to Z$ be a morphism in $\Ind_\kappa^\mu(\calC)$, and let $g_1,g_2: X \rightrightarrows Y$ be morphisms in $\Ind_\kappa^\mu(\calC)$. Assume for contradiction that $fg_1 = fg_2$ but $g_1 \neq g_2$. First, we may assume that $X \in \calC$. For in general, $X$ is a colimit $X = \varinjlim_i X_i$ of objects of $\calC$, and so $g_1 = g_2 \Leftrightarrow g_1|_{X_i} = g_2|_{X_i}$ for each leg of the diagram. Now, assuming that $X \in \calC$, we may assume that $Y \in \calC$. For $X$ is $\kappa$-presentable and $Y = \varinjlim_j Y_j$ is a $\kappa$-filtered colimit of objects of $\calC$, so there are morphisms $\bar g_1, \bar g_2: X \rightrightarrows Y_j$ such that $\iota_j \bar g_1 = g_1$ and $\iota_j \bar g_2 = g_2$, where $\iota_j: Y_j \to Y$ is the canonical inclusion, and thus $\bar g_1 \neq \bar g_2$. Now, assuming that $X, Y \in \calC$, we may assume that $Z \in \calC$. For $Y$ is $\kappa$-presentable and $Z = \varinjlim_k Z_k$ is a $\kappa$-filtered colimit of objects of $\calC$, so there is some $Z_k \to Z$ through which $f$ factors. Thus we may assume that the whole diagram is in $\calC$. But then we obtain a contradiction, since $f$ is a monomorphism in $\calC$.
\end{proof}

We thank Zhen Lin Low for pointing out the proof of the following \cite{mo-zl}.
\begin{fact}\label{fact:accprops-ind-ind}
Let $\calC$ be a small category, and let $ \kappa < \lambda_1$ be a regular cardinal. There is an equivalence of categories $F: \Ind_\kappa(\Ind^\kappa(\calC)) \to \Ind(\calC)$, natural in $\calC$.
\end{fact}
\begin{proof}
This follows from the results of \cite{zhen-lin}. See Example 1.8, Proposition 3.5, and Proposition 3.10 in \emph{op.~cit.} Be aware that the use of subscripts and superscripts in \emph{op.~cit.}, Definition 1.2 clashes with our own notation: the subscripts there correspond to our superscripts, while our subscripts have no direct analog in their notation. See \cite{mo-zl} for further discussion.
\end{proof}

\subsection{Background in homotopy theory}\label{sec:hty-bkgd}

We recall some basic concepts from the homotopy theory of categories which are necessary to formulate and prove Theorem \ref{thm:mainthm}. The discussion here will be terse and lacking in motivation. For general homotopy theory there are many references available, such as \cite{hatcher}. For simplicial homotopy theory, we recommend \cite{riehl} or \cite{goerss_jardine}. For the homotopy theory of categories specifically, see \cite{quillen}.

\begin{Def}[Classifying Space of a Category]\label{def:classifying-space}
Let $\calC$ be a moderate category. For $n \in \nats$, let $\calC_n$ denote the (moderate) set of paths of length $n$ in $\calC$. That is, an element of $\calC_n$ consists of a length-$n$ chain 
\[
\begin{tikzcd}
C_0\arrow[r,"f_0"]&\cdots\arrow[r,"f_n"]&C_n
\end{tikzcd}
\]
of composable morphisms in $\calC$. Note that $\calC_0$ is precisely the set of objects in $\calC$. The \defterm{classifying space} of $\calC$ is the topological space $|\calC| = \amalg_n \calC_n \times \Delta^n / \sim$, where each $\calC_n$ is equipped with the discrete topology. Here $\Delta^n = \{(x_0,\dots,x_n) \in \reals^{n+1} \mid 0 \leq x_i \leq \sum_i x_i = 1\}$ is the topological $n$-simplex, and the equivalence relation $\sim$ is generated by the following identifications:
\begin{itemize}
     \item $((f_1,\dots,f_{i-1},\id,f_{i+1}\dots,f_n),(x_0,\dots,x_n)) \sim \\ ((f_1,\dots,f_{i-1},f_{i+1},\dots,f_n), (x_0,\dots,x_{i-1}+x_{i},\dots,x_n))$.
     \item $((f_1,\dots,f_n),(x_0,\dots,x_{i-1},0,x_{i+1},\dots,x_n)) \sim \\ ((f_1,\dots,f_{i+1}f_i,\dots,f_n),(x_0,\dots,x_{i-1},x_{i+1},\dots,x_n))$.
\end{itemize}
When $i=0$ in the second bullet, ``$f_{i+1} f_i$" really means ``$f_1$", and likewise ``$f_{i+1}f_i$" really means ``$f_n$" when $i=n$. If $F: \calC \to \calD$ is a functor, there is an induced map $|F|: |\calC| \to |\calD|$. Thus the classifying space determines a functor $|-|: \Cat \to \Top$.\footnote{For technical reasons, one often restricts attention to some subcategory of all topological spaces, such as \emph{compactly-generated spaces}. But nothing in this paper depends strongly on this distinction.}.
\end{Def}
\begin{rmk}\label{rmk:nerve-realization}
For any category $\calC$, the classifying space $|\calC|$ is a CW complex. The vertices of $|\calC|$ are the objects of $\calC$. The edges of $|\calC|$ are the non-identity morphisms of $\calC$. The 2-faces of $|\calC|$ correspond to equations $gf = h$; for such an equation the three boundary edges of the corresponding face are $g$, $f$, and $h$. Higher faces correspond to longer equations $f_n \cdots f_1 = h$; the faces of such a face are obtained by composing various sub-strings of the list $(f_1,\dots,f_n)$.

From a more conceptual perspective, the classifying space functor $|-|: \Cat \to \Top$ factors through the category of \defterm{simplicial sets}. Specifically, we have $|\calC| = |\nerve \calC|$ where $\nerve: \Cat \to \sSet$ is the \defterm{nerve} functor and $|-| : \sSet \to \Top$ is the \defterm{geometric realization} functor (note that we use the same notation for geometric realization as we do for classifying spaces). The object $\nerve \calC$ is a \defterm{simplicial set} whose set $\calC_n$ of length-$n$ chains of composable morphisms as in Definition \ref{def:classifying-space}. For more on nerves, geometric realization, and classifying spaces, see \cite{quillen}. For geometric realization specifically, see \cite[Chapter 3]{may-simplicial}, and \cite[Chapter IX.6]{cwm} for a conceptual definition using coends.

In fact, for homotopy theoretic purposes there is really no need to perform the second step of geometric realization: one can ``do homotopy theory" directly at the level of simplicial sets rather than topological spaces. We have chosen to state our results in terms of topological spaces because simplicial sets may be less familiar to model theorists, but the reader who is familiar with simplicial sets is welcome to interpret our results in that setting instead.
\end{rmk}

We would like to study the space $|\calC|$ homotopy-theoretically. To that end, let us recall some basic definitions.

\begin{Def}[Homotopy]\label{def:homotopy}
Let $X,Y$ be topological spaces and $f,g : X \rightrightarrows Y$ a pair of maps. A \defterm{homotopy} from $f$ to $g$ is a map $F: [0,1] \times X \to Y$ such that $F(0,x) = f(x)$ and $F(1,x) = g(x)$ for all $x \in X$. Two maps are said to be \defterm{homotopic} if there is a homotopy between them; this is a congruence relation on $\Top$. The category of CW complexes and homotopy classes of maps is called \defterm{the homotopy category}. A \defterm{homotopy equivalence} from $X$ to $Y$ consists of maps $f: X \to Y$, $g: Y \to X$, and homotopies $gf \sim 1$, $fg \sim 1$. In this case we say that $X$ and $Y$ are \defterm{homotopy equivalent}, and we write $X \simeq Y$ if $X$ and $Y$ are homotopy equivalent CW complexes. This is an equivalence relation on CW complexes, and we say that two CW complexes have the same \defterm{homotopy type} if they are homotopy equivalent.
\end{Def}

\begin{fact}\label{fact:catequiv} $~$
\begin{enumerate}
    \item A natural transformation $\alpha: F \Rightarrow G: \calC \to \calD$ induces a homotopy $|F| \sim |G|$.
    \item\label{propitem:catequiv} An equivalence of categories $F: \calC \to \calD$ induces a homotopy equivalence $|F|: |\calC| \simeq |\calD|$.    \item\label{fact:catequiv.item:adj} More generally, an adjunction $F: \calC \to \calD: G$ induces a homotopy equivalence $|\calC| \simeq |\calD|$.
    \item\label{fact:catequiv.item:minus} More generally still, if $F: \calC \to \calD : G$ are functors and there are natural transformations between $\id_\calC$ and $GF$ and $\id_\calD$ and $FG$ (the direction doesn't matter), then $|\calC| \simeq |\calD|$.
\end{enumerate}
\end{fact}
\begin{proof}
One may think of a natural transformation $\alpha: F \Rightarrow G: \calC \to \calD$ as a functor $[1] \times \calC \to \calD$, where $[1]$ is the ``arrow category," with two objects and one nonidentity morphism between them. Similarly, a homotopy between maps $X\rightrightarrows Y$ is a map $\Delta^1 \times X \to Y$. From this, and the fact that the classifying space functor preserves finite products (cf. Fact \ref{fact:finprod} below), $(1)$ follows. Then $(4)$ is immediate; $(2)$ and $(3)$ are consequences of $(4)$.
\end{proof}

We next consider the simplest type of space up to homotopy:

\begin{Def}[Contractible]
A topological space $X$ is \defterm{contractible} if it is homotopy equivalent to the 1-point space $\Delta^0$. A category $\calC$ is \defterm{contractible} if its classifying space $|\calC|$ is contractible. In particular, a groupoid is contractible if and only if it is trivial (cf. Recollection \ref{rec:gpd}).
\end{Def}



\begin{Def}[Fundamental Groupoid of a Topological Space]\label{def:fundgpd}
Let $X$ be a topological space. The \defterm{fundamental groupoid} $\Pi_1(X)$ is the groupoid with objects the points of $X$ and morphisms given by homotopy classes (relative to endpoints) of paths, with composition given by concatenation of paths; inverses are given by reversing paths. If $x \in X$ is a point, the \defterm{fundamental group} $\pi_1(X,x)$ is the group of automorphisms of $x$ considered as an object of $\Pi_1(X)$. When $X$ is path-connected and nonempty, $\pi_1(X,x)$ is independent of $x$ up to isomorphism, so we may refer to $\pi_1(X)$ loosely.
\end{Def}

\begin{fact}\label{fact:topfundgpd} $~$
\begin{enumerate}
    \item\label{fact:topfundgpd.item:conn} Let $\calC$ be a moderate category. Then $\calC$ is connected if and only if $|\calC|$ is connected.
    \item\label{fact:topfundgpd.item:equiv} Let $\calC$ be a moderate category. The natural functor $\Pi_1(\calC) \to \Pi_1(|\calC|)$ is an equivalence of groupoids.
    \item\label{fact:topfundgpd.item:equiv-equiv} A homotopy equivalence $f: X \simeq Y$ of topological spaces induces an equivalence of groupoids $f_\ast: \Pi_1(X) \simeq \Pi_1(Y)$, and for any $x \in X$ induces an isomorphism of fundamental groups $f_\ast: \pi_1(X,x) \cong \pi_1(Y,f(x))$.
    \item In particular, if $X$ is a contractible topological space, then $\Pi_1(X)$ is trivial.
\end{enumerate}
\end{fact}
\begin{proof}
$(1)$ follows directly from the definitions. For $(2)$, note that the fundamental group of $|\calC|$ may be computed from the cell structure of its definition using the van Kampen theorem, and one obtains precisely the presentation of $\Pi_1(\calC)$ found in \cite[Chapter I.1]{gabriel-zisman}. $(3)$ and $(4)$ may be deduced from the corresponding familiar facts for the fundamental group.
\end{proof}
 
\begin{eg}[Classifying Space of Group]\label{eg:kgn}
Let $G$ be a discrete group. We denote by $BG := |\bbB G|$ the classifying space of the groupoid $\bbB G$, called the \defterm{classifying space} of the group $G$. In topology, the space $BG$ is defined more generally for a topological group $G$, but in this paper we will only consider it for discrete groups. To emphasize this, we may write $BG^\delta$ in cases where $G$ may also carry a topology which we are ignoring.
\end{eg}
 
\begin{rmk}[Homotopy Groups, Weak Homotopy Equivalence]\label{rmk:kgn}
Let $X$ be a space and $x \in X$ a point. For each $n \in \nats$, the \defterm{$n$th homotopy set of $X$ based at $x$}, denoted $\pi_n(X,x)$, is defined to be the set of pointed homotopy equivalence classes of maps $S^n \to X$ sending the base-point to $x$, where $S^n$ is the $n$-sphere. For $n \geq 1$, $\pi_n(X,x)$ is a group (recovering the fundamental group for $n=1$) and for $n \geq 2$ it is abelian. As for the fundamental group, choices of base-point in the same path component yield isomorphic homotopy groups, so if $X$ is connected we may speak loosely of $\pi_n(X)$. The construction $\pi_n(X,x)$ is functorial in base-point-preserving maps. A \defterm{weak homotopy equivalence} is a map which induces bijections on all homotopy sets for all base-points. Every homotopy equivalence is a weak homotopy equivalence. Every space is weak homotopy equivalent to a CW complex, and Whitehead's theorem says that a map between CW complexes is a homotopy equivalence if and only if it is a weak homotopy equivalence. We write $X \simeq Y$ if there is a diagram of weak homotopy equivalences $X \to Z \leftarrow Y$, and say that $X$ and $Y$ are \defterm{weakly homotopy equivalent}. This is an equivalence relation on spaces \cite[Proposition 4.13 and Corollary 4.19]{hatcher}, and when $X,Y$ are CW complexes we have $X \simeq Y$ if and only if $X$ and $Y$ are homotopy equivalent, agreeing with the notation introduced in Definition \ref{def:homotopy}. The only spaces considered in this paper which are not CW complexes are certain infinite products of CW complexes in Section \ref{sec:examples}, so most of the time the distinction between homotopy equivalence and weak homotopy equivalence is not important for us.

We also say that a functor $F: \calC \to \calD$ is a \defterm{weak homotopy equivalence} if it induces a homotopy equivalence (equivalently: a weak homotopy equivalence) of classifying spaces. 
\end{rmk}

\begin{rmk}[Aspherical]\label{rmk:kgn2}
In general, there exist CW complexes with isomorphic homotopy groups which are not homotopy equivalent. However, there is a special case of interest: we say that $X$ is \defterm{aspherical} if $\pi_n(X,x) = 0$ for all base-points and all $n\geq 2$. It is a fact that if $X$ and $Y$ are aspherical CW complexes, then any equivalence of fundamental groupoids $\Pi_1(X) \simeq \Pi_1(Y)$ is induced by a homotopy equivalence $X \simeq Y$ \cite[Chapter 4.2]{hatcher}. In particular, if $X$ is aspherical, then $X \simeq |\Pi_1(X)|$, and in particular if $X$ is connected and aspherical, then $X \simeq B \pi_1(X)$.\footnote{There is a generalization to higher homotopy groups: if $n\geq 1$ and $A$ is a group (required to be abelian if $n \geq 2$), then there is a unique space up to homotopy equivalence, called the \defterm{$n$th Eilenberg-MacLane space} for the group $A$, and denoted $K(A,n)$ or $B^n A$, satisfying the property that $\pi_k(B^n A) = \begin{cases} A & n=k \\ 0 & \text{else} \end{cases}$.}
\end{rmk}

Deeper tools in homotopy theory include Quillen's celebrated theorems A and B.\footnote{We would like to stress that although Quillen's Theorem A and B are often considered relatively deep theorems in homotopy theory, in this paper their use is merely a convenience until Section \ref{sec:higher-hty}, and all of our results up to that point can be proved directly by more elementary methods.} For the notions of opfibrations, slice categories, and reindexing functors, see Recollection \ref{rec:fibration}.

\begin{thm}[Quillen's Theorem A {\cite[Theorem A and Corollary immediately following]{quillen}}]\label{thm:quillen-a}
Let $F: \calC \to \calD$ be a Grothendieck opfibration in $\Cat$. Suppose that for all $D \in \calD$, the fiber $F\inv(D)$ is contractible. Then $|F|: |\calC| \to |\calD|$ is a homotopy equivalence.

Let $F: \calC \to \calD$ be any functor in $\Cat$. Suppose that for all $D \in \calD$, the slice category $F \downarrow D$ is contractible. Then $|F|: |\calC| \to |\calD|$ is a homotopy equivalence.
\end{thm}

\begin{thm}[Quillen's Theorem B {\cite[Theorem B and Corollary immediately following]{quillen}}, see also {\cite[Chapter IV.5]{goerss_jardine}}]\label{thm:quillen-b}
Let $F: \calC \to \calD$ be a Grothendieck opfibration in $\Cat$. Suppose that for every morphism $D \to D'$ in $\calD$, the reindexing functor $F\inv(D) \to F\inv(D')$ induces a homotopy equivalence of classifying spaces. Then the induced sequence $|F\inv(D)| \to |\calC| \to |\calD|$ is a homotopy fiber sequence for each $D \in \calD$.
\end{thm}

Because $|\calC| \cong |\calC^\op|$ naturally, these theorems also have dual forms. Theorem B also has a form using coslice categories analogous to Theorem A, but we will not need it.

In the statement of Theorem B, a \defterm{homotopy fiber sequence} is a sequence of maps $X \to Y \to Z$ such that the composite $X \to Z$ is constant and the induced square \[
\begin{tikzcd}[row sep = small, column sep = small]
X \ar[r] \ar[d] & Y \ar[d] \\
1 \ar[r] & Z
\end{tikzcd}\]
is \defterm{homotopy cartesian} in the sense of \cite[p. 96]{quillen}. For our purposes, it is not important to know precisely what this means. We simply need to know the following facts:

\begin{fact}\label{fact:qtb-cor}
Let $F: \calC \to \calD$ be a Grothendieck opfibration in $\Cat$, where $\calD$ is a groupoid. Then for any $D \in \calD$, the induced sequence $|F\inv(D)| \to |\calC| \to |\calD|$ is a homotopy fiber sequence.
\end{fact}
\begin{proof}
Any morphism in $\calD$ is an isomorphism, and so any cocartesian morphism of $\calC$ is an isomorphism. Thus the reindexing functors $F\inv(D) \to F\inv(D')$ are all equivalences of categories, and in particular induce homotopy equivalences of classifying spaces by Fact \ref{fact:catequiv}. Thus by Quillen's Theorem B (Theorem~\ref{thm:quillen-b}), the sequence $|F\inv(D)| \to |\calC| \to |\calD|$ is a homotopy fiber sequence.
\end{proof}

\begin{fact}[{\cite[Theorem 4.41]{hatcher}}]\label{fact:les}
Let $F \to E \to B$ a homotopy fiber sequence. Assume for simplicity that $F,E,B$ are path-connected, and choose a basepoint for $F$ (which maps forward to choices of basepoint for $E$ and $B$.) Then there is an induced long exact sequence of homotopy groups $ \dots \to \pi_n(F) \to \pi_n(E) \to \pi_n(B) \to \pi_{n-1}(F) \to  \dots$, which is natural with respect to basepoint-preserving maps of fiber sequences.
\end{fact}

\begin{fact}\label{fact:five-fiber}
Let
\begin{equation*}
\begin{tikzcd}
F \ar[r] \ar[d] & E \ar[r] \ar[d] & B \ar[d] \\
F' \ar[r] & E' \ar[r] & B'
\end{tikzcd}
\end{equation*}
be a commutative diagram of connected CW complexes, and suppose that both rows are homotopy fiber sequences. Then if any two of the downward maps are homotopy equivalences, so is the third.
\end{fact}
\begin{proof}
After choosing a basepoint for $F$ and mapping it forward to basepoints for the 5 other spaces, Fact \ref{fact:les} supplies two long exact sequences homotopy groups with natural maps between them. The hypotheses imply that all but every third map is an isomorphism. Therefore by the five lemma, all the maps are isomorphisms. By Whitehead's theorem (cf. Remark \ref{rmk:kgn}), all three maps are homotopy equivalences.
\end{proof}

\begin{fact}\label{fact:ses}
Let $1 \to G \to H \to K \to 1$ be a short exact sequence of groups. Then $BG \to BH \to BK$ is a homotopy fiber sequence.
\end{fact}
\begin{proof}
The functor $\bbB H\to \bbB K$ is a Grothendieck opfibration because both $\bbB H$ and $\bbB K$ are groupoids and the functor is full and surjective on objects. Then this follows from from Fact \ref{fact:qtb-cor}.
\end{proof}

The following statements are useful tools when computing examples in Section \ref{sec:examples}, and go back at least to \cite{quillen}.

 \begin{prop}\label{prop:contractible}
Let $\calC$ be a category. Then $\calC$ is contractible in all of the following cases:
\begin{enumerate}
    \item\label{prop:contractible.item:initial} $\calC$ has an initial or terminal object.
    \item\label{prop:contractible.item:slice} In particular, if $\calC = \calC'\downarrow C$ is a slice category.
    \item\label{prop:contractible.item:almostcoprod} $\calC$ is nonempty and admits a functor $F: \calC \times \calC \to \calC$ and natural transformations $\iota_1: \pi_1 \Rightarrow F$ and $\iota_2: \pi_2 \Rightarrow F$, where $\pi_1,\pi_2: \calC \times \calC \to \calC$ are the projection functors.
\end{enumerate}
\end{prop}
\begin{proof}
\begin{enumerate}
    \item Suppose that $\calC$ has an initial object. Then the unique functor $\calC \to 1$ to the terminal category (cf. Example \ref{eg:colim}) has a left adjoint, given by the inclusion of the initial object. Then by Fact \ref{fact:catequiv}(\ref{fact:catequiv.item:adj}), $| \calC|$ is contractible. The case of a terminal object is dual.
    \item If $\calC = \calC'\downarrow C$ is a slice category, then $(C,\id_C)$ is a terminal object in $\calC$.
    \item Suppose that $\calC$ is nonempty, pick some $C_0 \in \calC$, and define a functor $F(C_0,-): \calC \to \calC$, $C \mapsto F(C_0, C)$. There is a natural transformation $\id_\calC \Rightarrow F(C_0,-)$ whose components are given by $\iota_2: C \to F(C_0, C)$. There is also a natural transformation $\iota_1: \const_{C_0} \Rightarrow F(C_0,-)$ from the constant functor at $C_0$ to $F(C_0,-)$. By transitivity of the homotopy relation, $|\id_\calC|$ is homotopic to $|\const_{C_0}|$, and so $\calC$ is contractible.\qedhere
\end{enumerate}
\end{proof}

\begin{rmk}\label{rmk:skolem}
Note that $\Mod(T)$ has a terminal object only when $T$ has only one finite model without nontrivial automorphisms. And $\Mod(T)$ has an initial object if and only if the definable closure of the empty set is a model (this is much stronger than the existence of a prime model, for example ACF has a prime model but no initial model).
\end{rmk}

\begin{Def}\label{def:fun-jn}
In the situation of Proposition \ref{prop:contractible}(\ref{prop:contractible.item:almostcoprod}) we will say that $\calC$ has \defterm{functorial joint embedding}. A motivating example comes when $F: \calC \times \calC \to \calC$ is a binary coproduct functor.
\end{Def}

We close this preliminary section with a few more miscellaneous facts we will need later.

\begin{fact}[{\cite[Proposition 5.7]{thomason}}]\label{fact:poset}
Let $X$ be a CW complex. Then $X$ is homotopy equivalent to the classifying space of a poset.
\end{fact}
\begin{proof}
As shown in \cite{thomason}, there is a \defterm{model structure} (the \emph{Thomason model structure}) in the sense of \cite{quillen-homotopical} on $\Cat$ where the \defterm{weak equivalences} of the model category structure are those functors which are carried to homotopy equivalences by the classifying space functor $|-| : \Cat \to \Top$. Moreover the classifying space functor is a \emph{Quillen equivalence}. In the Thomason model structure every cofibrant object is a poset. Direct from the definition of a model category, every object is weakly equivalent to a cofibrant object. Quillen equivalences induce bijections of weak equivalence classes of objects. Every topological space is weakly equivalent to CW complex.
\end{proof}

\begin{fact}[cf. {\cite{quillen-homotopical}}]\label{fact:po-contr}
A pushout of contractible spaces along injective cellular maps is contractible.
\end{fact}
\begin{proof}
The fundamental group and homology of the pushout may be computed using the van Kampen theorem and the usual Mayer-Vietoris sequence for homology, and the conclusion then follows from the Homology Whitehead Theorem (which says that a map between simply-connected spaces inducing an isomorphism on integral homology is a weak homotopy equivalence -- see \cite{may-whitehead} for a nice proof).
\end{proof}

\begin{fact}\label{fact:fil-whe}
(Weak) homotopy equivalences of geometric realizations of simplicial sets are stable under filtered colimits of simplicial maps. In particular, if $\calC = \varinjlim_i \calC_i$ is a filtered colimit of categories $\calC_i$ with contractible classifying spaces, then $\calC$ also has a contractible classifying space.
\end{fact}
\begin{proof}
This is recovered from the stronger results of \cite{raptis-rosicky} or \cite{barnea-schlank}, or alternatively may be deduced from the classical theory of Kan's $\mathrm{Ex}^\infty$ functor \cite{kan}.
\end{proof}

\begin{fact}\label{fact:geom-colim}
The geometric realization functor $|-|: \sSet \to \Top$ preserves colimits.
\end{fact}
\begin{proof}
In fact, the geometric realization functor is a left adjoint \cite[Proposition I.2.2]{goerss_jardine}.
\end{proof}

\begin{fact}\label{fact:finprod}
Let $\calC, \calD$ be categories. Then the canonical continuous bijection $|\calC \times \calD| \to |\calC| \times |\calD|$ is a weak homotopy equivalence. Moreover, if $\calC$ or $\calD$ is finite, then this map is a homeomorphism.\footnote{In practice, for homotopy theorists today the classifying space functor is usually modified to take values, not in $\Top$, but in a convenient category of topological spaces \cite{convenient} such as the category of \emph{compactly-generated weak Hausdorff spaces}. With this modification, the map $|\calC \times \calD| \to |\calC| \times |\calD|$ becomes a homeomorphism for all categories $\calC, \calD$, so that we may simply say that $|-|$ preserves finite products. We have endeavored to suppress reliance on modern tools such as simplicial sets and convenient categories of spaces for the convenience of the reader who is not a topologist.}
\end{fact}
\begin{proof}
More generally, if $X,Y$ are simplicial sets (cf. Remark \ref{rmk:nerve-realization}), then (i) the continuous bijection $|X \times Y| \to |X| \times |Y|$ is a weak homotopy equivalence, and (ii) this map is a homeomorphism if either $X$ or $Y$ is \defterm{locally finite}, meaning that it has finitely many simplices in each degree (note that the nerve of a finite category is locally finite). The proof of (ii) is due to \cite[Theorem 2]{milnor}. For (i), observe that if $S^k \to |X| \times |Y|$ represents an element of $\pi_k(|X|\times |Y|)$, then because $S^k$ is compact, the projection $S^k \to |X| \times |Y| \to |X|$ factors through a finite subcomplex $|X'| \subseteq |X|$. Similarly, the other projection factors through a finite subcomplex $|Y'| \subseteq |Y|$. Thus the whole map factors through $|X'| \times |Y'|$ which, by (i), is homeomorphic to $|X' \times Y'|$. Since $|X' \times Y'|$ is a subcomplex of $|X \times Y|$, we obtain a lift $S^k \to |X \times Y|$, showing that the map $\pi_k(|X \times Y|) \to \pi_k(|X| \times |Y|)$ is surjective. Similarly, we may lift any nullhomotopy $D^k \to |X| \times |Y|$ to a nullhomotopy $D^k \to |X \times Y|$, so that the map $\pi_k(|X \times Y|) \to \pi_k(|X| \times |Y|)$ is injective. Thus the map $\pi_k(|X \times Y|) \to \pi_k(|X| \times |Y|)$ is bijective for all $k$, and so $|X \times Y| \to |X| \times |Y|$ is a weak homotopy equivalence. 
\end{proof}

\begin{fact}\label{fact:prodgpd}
The geometric realization functor commutes with arbitrary products of groupoids up to weak homotopy equivalence. The functor $\bbB : \Grp \to \Gpd$ also commutes with arbitrary products.
\end{fact}
\begin{proof}
The second statement is trivial. For the first, observe that homotopy groups of spaces commute with infinite products and groupoids are closed under products in $\Cat$, so that $\Pi_1(\prod |\Gamma_i|) = \prod \Pi_1(|\Gamma_i|) = \prod \Gamma_i = \Pi_1(|\prod \Gamma_i|)$ and all other homotopy groups vanish.
\end{proof}


\section{Model theoretic preliminaries}\label{sec:mod-bkgd}
In this section, we introduce the notion of a the Lascar group of a theory. Most of the materials can be found in  \cite{casanovas2001galois} and \cite{zie}. For basic notions in model theory, refer to \cite{tentziegler}. We fix some conventions for this section.
\begin{itemize}
     \item $T$ is a complete first-order theory in some language $L$.
     \item $\kappa$ is a infinite cardinal, $\kappa > |T|$.
     \item $\Mod(T)$ is the category of models of $T$, with elementary embeddings as morphisms. Recall that in this paper, all the models are assumed to be small, see Section \ref{sec:set-con}.
     \item $\Mod_\kappa(T)$ is the full subcategory of \defterm{$\kappa$-small} models of $T$, i.e. models of cardinality $<\kappa$.
 \end{itemize}
 \begin{Def}\label{def:univ-homog-etc}
 Let $\bbU \in \Mod(T)$ be a model of $T$.
 \begin{enumerate}
    \item $\bbU$ is \defterm{$\kappa$-universal} if for every $M \in \Mod_\kappa(T)$, there exists a morphism $M \to \bbU$ in $\Mod(T)$.
    \item $\bbU$ is \defterm{strongly $\kappa$-homogeneous} if for every $M \in \Mod_\kappa(T)$ and every pair of morphisms \begin{tikzcd} f,g: M \ar[r,shift left] \ar[r,shift right] & \bbU \end{tikzcd} there exists an \emph{automorphism} $\alpha: \bbU \to \bbU$ such that $\alpha f = g$.
    \item $\bbU$ is \defterm{$\kappa$-saturated} if for every type  $p \in S_n(A)$, where $|A|<\kappa$ and $|A|\subseteq \bbU$, there is $a \in \bbU$ such that $a \models p$. Note that $\kappa$-saturated implies $\kappa$-universal.
    \item $\bbU$ is a \defterm{monster model} if $\bbU$ is $\kappa$-saturated and strongly $\kappa$-homogeneous. Note that monster models exists for any theory $T$ and any $\kappa$, see \cite[Chapter 5.1]{Chang&Keisler} or \cite[Chapter 6.1]{tentziegler} for details.
 \end{enumerate}
 \end{Def}
 \begin{rmk}\label{rmk:univ-homog}
 Note that the definition of $\kappa$-universal/strongly $\kappa$-homogeneous makes sense even in an accessible category. In an accessible category $\calC$, we say that an object $\bbU$ is $\kappa$-universal if for any $\kappa$-presentable $M$, there is a morphism $M\to \bbU$ and we say it is strongly $\kappa$-homogeneous if if for every $\kappa$-presentable $M$ and every pair of morphisms \begin{tikzcd} f,g: M \ar[r,shift left] \ar[r,shift right] & \bbU \end{tikzcd} there exists an \emph{automorphism} $\alpha: \bbU \to \bbU$ such that $\alpha f = g$.
 \end{rmk}
 \begin{rmk} \label{rmk:term-homog}
 Our notion of \emph{strong $\kappa$-homogeneity} might more precisely be called \emph{strong model $\kappa$-homogeneity} since we do not ask for homogeneity over arbitrary subsets. However, for a $\kappa$-universal model $M$, the two notions above coincide. We thank the referee for the following proof.
 \begin{proof}
 Given two subsets $A_1,A_2\subset M$ that are isomorphic. We need to find $\sigma\in \Aut(M)$ such that $\sigma(A_1)=A_2$. By Downwards L\"owenheim-Skolem, one can find $M_i\preceq M$ such that $A_i\subseteq M_i$ for $i=1,2$. Take $N$ to be a small model (not necessarily contained in $M$) containing $M_1,M_2$ such that $A_1,A_2$ are conjugates in $N$. By $\kappa$-universality of $M$, let $N'$ denote an embedded copy of $N$ in $M$, we call the image of $M_1,M_2$ under this embedding $M'_1,M'_2$. Similarly, we denote the image of $A_1,A_2$ by $A'_1,A'_2$. By strong model $\kappa$-homogeneity of $M$, we have $\sigma_1,\sigma_2,\sigma_3\in \Aut(M)$ such that $\sigma_1(M_1)=M'_1$,  $\sigma_2$ fixes $N'$ set-wise and sends $\sigma_2(A'_1)=A'_2$ and $\sigma_3(M_2')=M_2$. Then for $\sigma=\sigma_3\sigma_2\sigma_1$, we have $\sigma(A_1)=\sigma_3\sigma_2\sigma_1(A_1)=A_2$.
 \end{proof}
 \end{rmk}
 
 \begin{Def}\label{def:lascar-group}
 Let $A\subseteq U$ be a set of parameters, where $U$ is a model of $T$, and $|A|<\kappa$ as in the previous definition. We define
 \begin{align*}
      \Lst(T,A,U)=\left\langle \Aut(U/N): A\subseteq N\prec U,|N|<\kappa \right\rangle
 \end{align*}
i.e. the group generated by the groups $\Aut(U/N)$, the automorphisms of $U$ that stabilize $N$ pointwise, where $N$ ranges over small elementary substructures of $U$ which contain $A$. It is immediate from the definition that the above group is a normal subgroup of $\Aut(U/A)$, the pointwise stabilizer of $A$. The \defterm{Lascar Group of $T$ over $A$ based at $U$} is defined to be 
\begin{align}
    \Gal_L(T,U/A):=\Aut(U/A)/\Lst(T,A,U).
\end{align}
For $a,b\in U$, we write $\Lstp(a/A)=\Lstp(b/A)$ if there is some $\sigma\in \Lst(T,A,U)$ such that $\sigma(a)=b$. When $A=\emptyset$, we write $\Gal_L(T,U)$ to mean $\Gal_L(T,U/\emptyset)$. In the language of Definition \ref{def:lascar-start}, $\Gal_L(T,U)=\Gal_L(\Mod(T),\Mod_\kappa(T),U)$. Our definition agrees with the classical definition when $U$ is sufficiently homogeneous and saturated.
 \end{Def}

 \begin{rmk}
 It is well-known that $\Gal_L(T,U)$ is independent of the choice of $U$ up to isomorphism so long as certain conditions on $U$ are met. In particular, although stronger hypotheses on $U$ were originally required in \cite{lascar}, it is shown in \cite[Theorem 4.3]{resplendent} that $\Gal_L(T,U)$ is independent of $T$ up to isomorphism so long as $U$ is $|T|^+$-universal and strongly $|T|^+$-homogeneous. We will denote this group by $\Gal_L(T)$, and it is also easy to see that $\Gal_L(T)$ has cardinality at most $2^{|T|}$. We will provide new proofs of the above facts in the course of Section \ref{sec:lascfund}, Corollary \ref{cor:lascar-bounded} and Corollary \ref{cor:lascar-size}.
 \end{rmk}
 
 \begin{rmk}
 Note that the Lascar group has a different definition via the Lascar strong types. For each sort $S$, there is a finest bounded invariant equivalence relation $E^S_L$ on the sort $S$. Take $\bbU$ to be a monster model, then the Lascar strong automorphisms $\Lst(\bbU)$ are the ones that fixes each such class of $E^S_L$, and $\Gal_L(\bbU)=\Aut(\bbU)/\Lst(\bbU)$. However, in this paper, we use Definition \ref{def:lascar-group}.
 \end{rmk}
 

\section{The Lascar Group as a Fundamental Group}\label{sec:lascfund}
In this section, let $T$ be a complete first-order theory, let $\kappa >|T| $ be a regular cardinal, and let $\bbU \in \Mod
(T)$ be a model which is $\kappa$-universal and strongly $\kappa$-homogeneous. The results of this section extend, with the same proofs, to any $\kappa$-accessible category containing a $\kappa$-universal, strongly $\kappa$-homogeneous object in the sense of Remark~\ref{rmk:univ-homog}. We need the following facts.

 \begin{prop}\label{prop:smallequiv}$~$
 \begin{enumerate}
    \item\label{prop:smallequiv.item:zero} Let $M \in \Mod(T)$. Then the slice category $\Mod_\kappa(T) \downarrow M$ is $\kappa$-filtered.
     \item\label{prop:smallequiv.item:one} The inclusion $\Mod_\kappa(T) \to \Mod(T)$ induces a homotopy equivalence $|\Mod_\kappa(T)|\allowbreak \to |\Mod(T)|$.
     \item\label{prop:smallequiv.item:two} $\Mod(T)$ and $\Mod_\kappa(T)$ are connected.
     \item\label{prop:smallequiv.item:three} The homomorphism of groupoids induced by inclusion of $\Mod_\kappa(T)\to \Mod(T)$, $\Pi_1( \Mod_\kappa(T)) \to \Pi_1(\Mod(T))$ is an equivalence.
 \end{enumerate}
 \end{prop}
 \begin{proof}
    Note that (\ref{prop:smallequiv.item:zero}) follows from the Downward L\"owenheim-Skolem Theorem. So by Fact \ref{fact:fil-whe}, $\Mod_\kappa(T) \downarrow M$ is contractible. By Quillen's Theorem A (Theorem \ref{thm:quillen-a}), the map $| \Mod_\kappa(T)| \to | \Mod(T)|$ is a homotopy equivalence, yielding (\ref{prop:smallequiv.item:one}). (\ref{prop:smallequiv.item:two}) follows from the joint embedding property (\ref{eg:modt-jep-ap}). (\ref{prop:smallequiv.item:three}) follows from (\ref{prop:smallequiv.item:one}) by Fact \ref{fact:topfundgpd}(\ref{fact:topfundgpd.item:equiv-equiv}).
  \end{proof}
 
 \begin{rmk}\label{rmk:ind-eq}
 Proposition \ref{prop:smallequiv}(\ref{prop:smallequiv.item:zero}), (\ref{prop:smallequiv.item:one}), and (\ref{prop:smallequiv.item:three}) continue to hold in greater generality: in place of $\Mod(T)$ we may take any connected $\kappa$-accessible category $\calC$, while in place of $\Mod_\kappa(T)$ we may take the full subcategory $\calC_\kappa$ of $\kappa$-presentable objects. In the case of (\ref{prop:smallequiv.item:zero}), see for example \cite[Proposition 2.8(ii)]{lpacc}. In the case of (\ref{prop:smallequiv.item:one}) and (\ref{prop:smallequiv.item:three}) the argument is essentially the same as above.
 \end{rmk} 

 \begin{prop}\label{prop:can-hom}
 For any $U \in \Mod(T)$, there is a canonical homomorphism $\phi:\Gal_L(T, U) \to \allowbreak \pi_1(\Mod(T), \allowbreak U)$ sending $[\alpha] \mapsto \llbracket \alpha \rrbracket$.
 \end{prop}
 In the statement of the proposition, we have $\alpha \in \Aut(U)$. The notation $[\alpha]$ denotes the class of $\alpha$ in $\Gal_L(T,U)$. The notation $\llbracket \alpha \rrbracket$, introduced in Recollection \ref{rec:catfund}, denotes the class of $\alpha$ in $\Pi_1(\Mod(T))$, and we are abusively identifying $\pi_1(\Mod(T),U)$ with its image under the inclusion $\bbB \pi_1(\Mod(T),U) \subseteq \Pi_1(\Mod(T))$.
 \begin{proof}
 The map $\phi:\alpha \mapsto \llbracket \alpha\rrbracket$ clearly defines a homomorphism $\Aut(U) \to \allowbreak \pi_1(\Mod(T), \allowbreak U)$; we just have to show that this descends through the quotient map $\Aut(U) \to \Gal_L(T,U)$ defining the Lascar group. For this, it suffices to show that if $\alpha$ is an automorphism of $U$ fixing a $\kappa$-small elementary sub-model $M \subseteq U$, then $\llbracket \alpha \rrbracket$ is trivial in $\pi_1(\Mod(T),U)$. This is true because if $i: M \to U$ is the inclusion, we have $\alpha  i = i$. This equation still holds in $\Pi_1(\Mod(T))$, but here $\llbracket i \rrbracket$ is invertible, more explicitly, $\llbracket \id \rrbracket = \llbracket i \rrbracket  \llbracket i \rrbracket \inv = \llbracket \alpha i \rrbracket \llbracket i \rrbracket\inv = \llbracket \alpha \rrbracket  \llbracket i \rrbracket  \llbracket i \rrbracket \inv = \llbracket \alpha \rrbracket$.
 \end{proof}

\begin{thm}\label{thm:mainthm}
The Lascar group $\Gal_L(T)$ is isomorphic to the fundamental group $\pi_1(\Mod(T))$. Specifically, if $\bbU$ is $\kappa$-universal and strongly $\kappa$-homogeneous, then the homomorphism $\phi: \Gal_L(T,\bbU) \to \pi_1(\Mod(T),\bbU)$ of Proposition \ref{prop:can-hom} is an isomorphism.
\end{thm}
\begin{proof}
The inclusion $F: \Pi_1(\Mod_\kappa(T)) \to \Pi_1(\Mod(T))$ is an equivalence by Proposition \ref{prop:smallequiv}. Let $G: \Pi_1(\Mod(T)) \to \Pi_1(\Mod_\kappa(T))$ be a weak inverse to $F$. Following the general construction of weak inverses to equivalences of categories from Recollection \ref{rec:eqcats}, $G$ depends on an arbitrary choice, for each $V \in \Mod(T)$, of an object $G(V) \in \Mod_\kappa(T)$ and an isomorphism $i_V: F(G(V)) \to V$ in $\Pi_1(\Mod(T))$. For convenience, we assume that we have chosen $G(\bbU)$ to be a $\kappa$-small sub-model $M_\ast$ with inclusion map $i_{M_\ast}: M_\ast \subseteq \bbU$ and that we have chosen $i_\bbU = \llbracket i_{M_\ast} \rrbracket$.

We define a functor $\tilde\phi: \bbB \Gal_L(T,\bbU) \to \Pi_1(\Mod_\kappa(T))$ to be the following composite:
\[\tilde \phi: \bbB \Gal_L(T,\bbU) \overset{\bbB \phi}{\to}
\bbB \pi_1(\Mod(T), \bbU) \overset{\sim}{\to} 
\Pi_1(\Mod(T)) \overset{\overset{G}{\sim}}{\to}
\Pi_1(\Mod_\kappa(T))\]
The first functor is induced by $\phi$ as in Recollection \ref{rec:gpd}. The second functor is the canonical inclusion, which is an equivalence as discussed in Recollection \ref{rec:conngpd} because $\Mod(T)$ is connected (Proposition \ref{prop:smallequiv}), and hence $\Pi_1(\Mod(T))$ is connected (Recollection \ref{rec:catfund}). The third functor is the weak inverse just constructed. Using the description of $G$ from Recollection \ref{rec:eqcats}, we see that $\tilde \phi$ is defined by the formula $\tilde \phi([\alpha]) = \llbracket i_{M_\ast}\rrbracket\inv \llbracket \alpha \rrbracket \llbracket i_{M_\ast} \rrbracket$.

The statement of the theorem (that $\phi$ is an isomorphism) equivalent to the statement that $\bbB \phi$ is an equivalence (Recollection \ref{rec:gpd}). Because $\tilde \phi$ is the composite of $\bbB \phi$ with an equivalence, this in turn is equivalent to the statement that $\tilde \phi$ is an equivalence, which we now show by constructing an explicit weak inverse.

\paragraph{Construction of inverse.}
We construct a weak inverse $\psi: \Pi_1(\Mod_\kappa(T)) \allowbreak \to \allowbreak \bbB \Gal_L(T, \bbU)$ to $\tilde \phi$. By the universal property of $\Pi_1$, we take $\psi$ to be the unique functor extending a functor $\Psi: \Mod_\kappa(T)\allowbreak \to \allowbreak \bbB \Gal_L(T,\bbU)$ defined as follows. Use the $\kappa$-universality of $\bbU$ to fix embeddings $i_M: M \to \bbU$ for each $\kappa$-small model $M$, noting that the morphism $i_{M_\ast}$ was fixed earlier in the proof. For convenience we take $i_M$ to be a sub-model inclusion $M \subseteq \bbU$ whenever $M$ is a sub-model of $\bbU$, but the embeddings $i_M$ are otherwise arbitrary. For $f: M \to N \in \Mod_\kappa(T)$ and $\alpha \in \Aut(\bbU)$, we say that $\alpha$ is a \defterm{compatible lift} of $f$ if $\alpha i_M = i_N f$. By strong $\kappa$-homogeneity of $\bbU$, for every morphism of $f: M \to N \in \Mod_\kappa(T)$ we may choose a compatible lift $\alpha_f$. We define $\Psi(f) := [\alpha_f]$. Note that if $\alpha,\beta$ are two compatible lifts of $f: M \to N$, then $\beta\inv \alpha$ fixes the model $i_M(M)$, and therefore $[\beta] = [\alpha]$ by definition of $\Gal_L(T)$. Thus we have $\Psi(f) = [\alpha]$ for \emph{any} compatible lift of $f$, and when computing $\Psi(f)$ we have the freedom to use whichever compatible lift is most convenient. For example, we show that $\Psi$ is a functor as follows. If $f = \id_M$ is an identity, then $\id_{\bbU}$ is a compatible lift, so $\Psi(\id_M) = [\id_\bbU]$ is the identity. And if $M \overset{f}{\to} N \overset{g}{\to} P$ is a composable pair and $\alpha,\beta$ are compatible lifts of $f,g$ respectively, then $\beta \alpha$ is a compatible lift of $gf$, so that $\Psi(gf) = [\beta\alpha] = [\beta][\alpha] = \Psi(g)\Psi(f)$.

\paragraph{Verification of inverse.} It remains to check that $\psi$ is a weak inverse to $\tilde \phi$.

First we check that $\tilde \phi \psi$ is naturally isomorphic to the identity on $\Pi_1(\Mod_\kappa(T))$. As noted in Recollection \ref{rec:catfund}, it suffices to show that $\tilde \phi \Psi \cong \llbracket - \rrbracket$, where $\llbracket - \rrbracket: \Mod_\kappa(T) \to \Pi_1(\Mod_\kappa(T))$ is the canonical functor $f \mapsto \llbracket f \rrbracket$. For this, define $\iota: \llbracket - \rrbracket \Rightarrow \tilde \phi \Psi$ by the formula $\iota_M = G(\llbracket i_{M_\ast} \rrbracket \inv \llbracket i_M \rrbracket)$. Note here that $\llbracket i_{M_\ast} \rrbracket \inv$ is well-defined because $\Pi_1(\Mod(T))$ is a groupoid, and for the same reason $\iota_M$ is an isomorphism. It is easy to check that $\iota$ is a natural isomorphism.

Now we check that $\psi \tilde \phi$ is equal to the identity on $\bbB \Gal_L(T)$, i.e. that for each $\alpha \in \Aut(\bbU)$ we have $\psi(\tilde \phi([\alpha])) = [\alpha]$. We have $\tilde \phi([\alpha]) = \llbracket i_{M_\ast} \rrbracket\inv \llbracket \alpha i_{M_\ast} \rrbracket$, where the composition takes place in $\Pi_1(\Mod(T))$. By Downward L\"owenheim-Skolem, there is a $\kappa$-small sub-model $N \overset{i_N}{\to} \bbU$ containing both $M_*$ and $\alpha M_*$. Let $j: M_\ast \to N$ denote the inclusion and $k = \alpha j: M_\ast \to N$. Then we have
    \begin{align*}
        \tilde \phi([\alpha])
        &= \llbracket i_{M_\ast} \rrbracket\inv \llbracket \alpha i_{M_\ast} \rrbracket \\
        &= \llbracket i_N j \rrbracket\inv \llbracket i_N k \rrbracket \\
        &= \llbracket j \rrbracket\inv \llbracket i_N \rrbracket\inv \llbracket i_N \rrbracket \llbracket k \rrbracket \\
        &= \llbracket j \rrbracket \inv \llbracket k \rrbracket
    \end{align*}
    Now the composition takes place in $\Pi_1(\Mod_\kappa(T))$, so we have 
    \begin{align*}
    \psi(\tilde \phi([\alpha])) = \psi(\llbracket j \rrbracket)\inv \psi(\llbracket k \rrbracket)
    \end{align*}
    by functoriality of $\psi$, which is equivalently $\Psi(j)\inv \Psi(k)$. To compute this, we may choose $\id_\bbU$ as a compatible lift for $j$ and $\alpha$ as compatible lift for $k$ to see that $\psi(\tilde \phi([\alpha])) = [\id_\bbU]\inv [\alpha] = [\alpha]$.
\end{proof}

From Theorem \ref{thm:mainthm}, we recover some known facts about the Lascar group:

\begin{cor}[{\cite[Theorem 4.3]{resplendent}}]\label{cor:lascar-bounded}

$\Gal_L(T,\bbU)$ is independent, up to isomorphism, of the choice of a $|T|^+$-universal and strongly $|T|^+$-homogeneous model $\bbU$. 
\end{cor}
\begin{proof}
By Theorem \ref{thm:mainthm}, $\Gal_L(T,\bbU) \cong \pi_1(\Mod(T),\bbU)$ by taking $\kappa = |T|^+$. Moreover, $T$ is complete, so $\Mod(T)$ is a connected category (Proposition \ref{prop:smallequiv}(\ref{prop:smallequiv.item:two})), hence $\Pi_1(\Mod(T))$ is a connected groupoid (as noted in Recollection \ref{rec:catfund}). As remarked in Recollection \ref{rec:conngpd}, this implies that the groups $\pi_1(\Mod(T),\bbU)$ are all isomorphic for different $\bbU$.
\end{proof}
\begin{rmk}
Note that in our theorem, the homogeneity condition on $\bbU$ is that it needs to be $|T|^+$-homogeneous with respect to sub-models, but as pointed out in Remark \ref{rmk:term-homog} by the referee, the two notions coincide when $\bbU$ is $|T|^+$-universal. Thus our hypotheses are equivalent to the hypotheses found in \cite[Theorem 4.3]{resplendent}.
\end{rmk}
\begin{rmk}
Here is a slight variation of the proof of Corollary \ref{cor:lascar-bounded}. Again by Theorem \ref{thm:mainthm} we have $\Gal_L(T,\bbU) \cong \pi_1(\Mod(T),\bbU)$. In turn we have $\pi_1(\Mod(T),\bbU) \cong \pi_1(|\Mod(T)|,\bbU)$. Because $\Mod(T)$ is connected, $|\Mod(T)|$ is path-connected. Therefore $\pi_1(|\Mod(T)|,\bbU)$ is independent of $\bbU$ up to isomorphism, and the result follows.

Thus we deduce the independence of $\Gal_L(T,\bbU)$ from $\bbU$ directly from the fact that for a path-connected space $X$, the fundamental group $\pi_1(X,x)$ is independent of the base-point $x$ up to isomorphism. Of course, this standard topological fact is essentially equivalent to the fact from Recollection \ref{rec:conngpd} used above, namely that if $\Gamma$ is a connected groupoid, then $\pi_1(\Gamma, x)$ is independent of $x$ up to isomorphism.
\end{rmk}

\begin{cor}[{\cite[Theorem 43]{lascar}}]\label{cor:lascar-size}
$\Gal_L(T)$ has cardinality less than or equal to $2^{|T|}$. 
\end{cor}
\begin{proof}
By Theorem \ref{thm:mainthm} and Proposition \ref{prop:smallequiv}, we have $\bbB\Gal_L(T) \simeq \Pi_1(\Mod_\kappa(T))\simeq \sk\Pi_1(\Mod_\kappa(T)) \simeq\Pi_1(\sk \Mod_\kappa(T)$. Therefore it suffices to bound the size of a skeleton of the category $\Mod_\kappa(T)$. Taking $\kappa$ to be $|T|^+$, it is clear that a skeleton of $\Mod_\kappa(T)$ is has at most $2^{|T|}$ many objects, namely the non-isomorphic models of cardinality at most $|T|$. And the Hom sets are bounded by $2^{|T|}$ as well. Thus we have $\#(\sk\Pi_1(\Mod_\kappa(T))) \leq \max(\aleph_0, \#(\Mod_\kappa(T)))$. Hence the corollary follows.
\end{proof}

\begin{rmk}
Let $T$ be an incomplete first-order theory. The results of this section may be applied as follows. The category $\Mod(T)$ is the disjoint union of the categories $\Mod(T')$ for each completion $T'$ of $T$. Then likewise $\Pi_1(\Mod(T)) = \amalg_{T'} \Pi_1(\Mod(T'))$. We have that $\Pi_1(\Mod(T)) \simeq \amalg_{T'} \Gal_L(T')$.
\end{rmk}
 
\begin{rmk}
Theorem \ref{thm:mainthm} suggests that we may define the \defterm{Lascar groupoid} of an arbitrary category $\calC$ to simply be its fundamental groupoid $\Pi_1(\calC)$. Note that by Theorem \ref{thm:mainthm}, the Lascar group agrees with $\pi_1(\calC,M)$ for any $M\in \calC$ when $\calC = \Mod(T)$. The same proof also yields a description of $\Pi_1(\calC)$ similar to the usual description of the Lascar group as in Definition \ref{def:lascar-start} when $\calC$ is an Abstract Elementary Class with joint embedding, amalgamation, and no maximal models \cite{shelah2009classification}. 
\end{rmk}


 \section{Examples}\label{sec:examples}
 In this section, we illustrate the shift in perspective afforded by Theorem \ref{thm:mainthm} by computing $\Gal_L(T)$ for some familiar theories $T$. In fact, we do a bit more: we show that in all of these examples, $|\Mod(T)|$ is aspherical and thus, as discussed in Remark \ref{rmk:kgn2}, the entire homotopy type of $|\Mod(T)|$ is characterized by its fundamental group $\pi_1(|\Mod(T)|) \cong \Gal_L(T)$.
 
 \begin{eg}[Sets]\label{eg:sets}
 Let $T$ be a complete theory in the empty language. Then $T$ is the theory of a set with $n$-many elements, for some $n \in \nats \cup \{\infty\}$. In the finite case, $\Mod (T)$ is a groupoid equivalent to $\bbB \Sigma_n$ where $\Sigma_n$ is the the $n$th symmetric group; see Example \ref{eg:kgn}, its classifying space is $B \Sigma_n$ and in particular $\Gal_L(T) \cong \Sigma_n$. In the infinite case, $\Mod(T)$ has functorial joint embedding given by disjoint union (see Definition \ref{def:fun-jn}), so $\Mod(T)$ is contractible by Proposition \ref{prop:contractible}(\ref{prop:contractible.item:almostcoprod}). Thus, in general $\Mod(T) \simeq \bbB \Sigma_n$ and $\Gal_L(T) \cong \Sigma_n$, where we adopt the convention that $\Sigma_n$ is trivial when $n = \infty$.
 
 More precisely, there is a functor $F: \Mod(T) \to \bbB \Sigma_n$  such that $|F|$ is a homotopy equivalence, defined as follows. If $n = \infty$ there is a unique such functor as we have the convention that $\Sigma_\infty$ is trivial. If $n<\infty$ we arbitarily label each model of $T$ with the elements $\{1,\dots,n\}$. Each elementary embedding $\phi$ permutes the labels, thus yielding a well-defined element $F(f) \in \Sigma_n$, and this assignment is functorial. 
 \end{eg}
 
 \begin{eg}[$\kappa$-sorted Sets]\label{eg:finsortsets}
 Let $T$ be a complete theory in the empty language over $\kappa$-many sorts. Then for each sort $\alpha$, there is some $n_\alpha \in \nats \cup \{\infty\}$ such that $T$ says there are $n_\alpha$-many elements of sort $\alpha$. The category $\Mod (T)$ is just a product of the categories of models of each sort individually.
 Taking classifying spaces preserves finite products up to weak homotopy equivalence (Fact \ref{fact:finprod}), but not infinite products in general. But we have $\Mod(T) \simeq (\prod_{n_\alpha < \infty} \bbB \Sigma_{n_\alpha}) \times (\prod_{n_\alpha = \infty} \Mod(T'))$ where $T'$ is the theory of an infinite set, and so $|\Mod(T)| \simeq |\prod_{n_\alpha < \infty} \bbB \Sigma_{n_\alpha}| \times |\prod_{n_\alpha = \infty} \Mod(T')|$. On the one hand, taking classifying spaces does preserve infinite products of groupoids up to weak homotopy equivalence (Fact \ref{fact:prodgpd}), so the first factor is weak homotopy equivalent to $\prod_{n_\alpha < \infty} B \Sigma_{n_\alpha}$. On the other hand, a product of categories with functorial joint embedding (cf. Example \ref{eg:sets}) again has functorial joint embedding, so the second factor is contractible. Thus $|\Mod(T)| \simeq \prod_{\alpha} B\Sigma_{n_\alpha}$; in particular $\Gal_L(T) \cong \prod_{\alpha} \Sigma_{n_\alpha}$ (recall our convention from Example \ref{eg:sets} that $\Sigma_k$ is trivial when $k= \infty$).
 
More precisely, there is a functor $F: \Mod(T) \to \prod_\alpha \bbB \Sigma_{n_\alpha}$ such that $|F|$ is a homotopy equivalence, defined by arbitrarily labeling each finite sort $\alpha$ with the set $\{1,\dots n_\alpha\}$, for each $M \in \Mod(T)$, and taking the induced maps on labelings.
 \end{eg}

For the next example, we first recall the classification of complete theories $T$ of an equivalence relation. For each $n \in \nats$, there is $m(n) \in \nats \cup \{\infty\}$ such that $T$ says there are $m(n)$ many equivalence classes of size $n$. There are two cases: in case $(a)$, there are $m(n) \neq 0$ for arbitrarily large $n$, while in case $(b)$, we have $m(n) = 0$ for $n$ sufficiently large. Either way, $T$ admits quantifier elimination once we add predicates $p_n(x)$ saying that the equivalence class of $x$ is of size $\geq n$. In case $(a)$ the number of infinite equivalence classes is arbitrary, while in case $(b)$, $T$ says that there are $m(\infty)$ many infinite equivalence classes for some $m(\infty) \in \nats \cup \{\infty\}$. An elementary embedding between models of $T$ is simply an injection preserving the equivalence relation and the size of each finite equivalence class.

\begin{eg}[An equivalence relation]
Let $T$ be the complete theory of an equivalence relation, and define $m: \nats \cup \{\infty\} \to \nats \cup \{\infty\}$ as above. In case $(a)$ as above, the inclusion $\calC \to \Mod(T)$, of the full subcategory $\calC$ of models with no infinite equivalence classes, has a right adjoint given by deleting all the infinite equivalence classes. Thus $|\Mod(T)| \simeq |\Mod(\calC)|$ by Fact \ref{fact:catequiv}(\ref{fact:catequiv.item:adj}). Therefore in case $(a)$ we will simply replace $\Mod(T)$ with $\calC$ and set $m(\infty) = 0$. We will abusively refer to $\calC$ as $\Mod(T)$, so from now on there is no need to distinguish between case $(a)$ and case $(b)$.

For each $M \in \Mod(T)$ and each $n \in \nats \cup \{\infty\}$ with $m(n) < \infty$, arbitrarily label the equivalence classes of size $n$ with the numbers $\{1,\dots, m(n)\}$, and if $m(n)< \infty$ and $n < \infty$ both hold, label also the elements of each equivalence class of size $n$ with the numbers $\{1,\dots,n\}$.  Any elementary embedding induces a bijection of labels, so we obtain a functor $F: \Mod(T) \to \prod_{m(n) < \infty} \bbB (\Sigma_{n} \wr \Sigma_{m(n)})$ (recall our convention from Example \ref{eg:sets} that when $n = \infty$, $\Sigma_n$ is trivial). Here $\Sigma_{n} \wr \Sigma_{m(n)}$ is the wreath product $\Sigma_{n}^{\times m(n)} \rtimes \Sigma_{m(n)}$, and the product is over all $n \in \nats \cup \{\infty\}$ such that $m(n) < \infty$. We claim that this functor induces an equivalence of classifying spaces. To see this, consider the following diagram:

\begin{tikzcd}
\Mod(T') \times \Mod(T'') \ar[r] \ar[d,"\sim"] & \Mod(T) \ar[r,"\pi F"] \ar[d,"F"] & \prod_{m(n) < \infty} \bbB \Sigma_{m(n)} \ar[d, equal] \\
\prod_{m(n) < \infty} \bbB \Sigma_{n}^{\times m(n)} \ar[r] & \prod_{m(n) <\infty} \bbB (\Sigma_{n} \wr \Sigma_{m(n)}) \ar[r,"\pi"] & \prod_{m(n) < \infty} \bbB \Sigma_{m(n)}
\end{tikzcd}

The bottom row arises from applying the functor $\bbB$ to a short exact sequence of groups, so it gives rise to a homotopy fiber sequence of classifying spaces by Fact \ref{fact:ses}. The functor $F$ is the one described above, and so the functor $\pi F$ is the functor which forgets all but the labelings of the equivalence classes. Its fiber is canonically isomorphic to $\Mod(T') \times \Mod(T'')$. Here $T'$ is the theory with a sort $S_{n,i}$ for each $n \in \nats \cup \{\infty\}$ such that $m(n) < \infty$, and each $1 \leq i \leq m(n)$, which says that $S_{n,i}$ has $n$-many elements. $T''$ is the theory with a sort $S_n$ for each $n \in \nats \cup \{\infty\}$ such that $m(n) = \infty$ and an equivalence relation on $S_n$ with infinitely many classes, all with $n$ elements. The leftmost downward functor is given by first projecting onto $\Mod(T')$ and then applying the homotopy-equivalence-inducing functor of Example \ref{eg:finsortsets}. The category $\Mod(T'')$ has functorial joint embedding given by disjoint union (cf. Definition \ref{def:fun-jn}), and so its classifying space is contractible by Proposition \ref{prop:contractible}(\ref{prop:contractible.item:almostcoprod}). Thus the leftmost downward arrow is the composite of two homotopy-equivalence-inducing functors and so also induces a homotopy equivalence. Since the leftmost and rightmost downward arrows induce homotopy equivalences of classifying spaces, by Fact \ref{fact:five-fiber} in order to show that $F$ also induces a homotopy equivalence it will suffice to show that the the rows induce homotopy fiber sequences of classifying spaces. We have already noted this for the bottom row, so it remains to show that the top row induces a homotopy fiber sequence of classifying spaces.

To this end, we invoke Quillen's Theorem B (Theorem~\ref{thm:quillen-b}) in the guise of Fact \ref{fact:qtb-cor}. The functor $\pi F$ is a Grothendieck opfibration. For if $\sigma \in \prod_{m(n) < \infty} \Sigma_{m(n)}$ and $M \in \Mod(T)$, then there is an automorphism $\phi$ of $M$ which permutes the labels in the manner described by $\sigma$ (recall from Recollection \ref{rec:fibration} that a lift of an isomorphism is (co)cartesian iff it is an isomorphism). By Fact \ref{fact:qtb-cor}, $\Mod(T') \times \Mod(T'') \to \Mod(T) \to \prod_{m(n) < \infty} \bbB \Sigma_{m(n)}$ induces a homotopy fiber sequence of classifying spaces. Thus $F$ induces a homotopy equivalence of classifying spaces.

Finally, the classifying space functor commutes with arbitrary products of groupoids up to weak homotopy equivalence (Fact \ref{fact:prodgpd}). Thus we have $|\Mod(T)| \simeq \prod_{m(n) < \infty} B (\Sigma_{n} \wr \Sigma_{m(n)})$. In particular, $\Gal_L(T) \cong \prod_{m(n) < \infty } \Sigma_{n} \wr \Sigma_{m(n)}$ (recall our convention from Example \ref{eg:sets} that $\Sigma_k$ is trivial when $k = \infty$).
\end{eg}
 
 \begin{eg}[Dense Linear Orders]
 Let $T$ be the theory of infinite dense linear orders without endpoints. Then $\Mod(T)$ has functorial joint embedding (see Definition \ref{def:fun-jn}). For, given models $M$ and $N$ we can form a new model placing $M$ above $N$. This construction is functorial, and admits natural embeddings of $M$ and $N$. Hence by Proposition \ref{prop:contractible}(\ref{prop:contractible.item:almostcoprod}) $|\Mod(T)|$ is contractible and in particular $\Gal_L(T)$ is trivial.
 \end{eg}
 
 \begin{eg}[Torsion-free Divisible Abelian Groups]
 Let $T$ be the theory of torsion-free divisible abelian groups. Then $\Mod(T)$ has functorial joint embedding (see Definition \ref{def:fun-jn}) using the direct sum functor. Hence by Proposition \ref{prop:contractible}(\ref{prop:contractible.item:almostcoprod}), $|\Mod(T)|$ is contractible and in particular $\Gal_L(T)$ is trivial.
 \end{eg}
 
 \begin{eg}[Divisible Ordered Abelian Groups]\label{eg:doag}
 Let $T$ be the theory of divisible ordered abelian groups. Then $\Mod(T)$ has functorial joint embedding given by taking the direct sum under the lexicographic ordering. Hence by Proposition \ref{prop:contractible}(\ref{prop:contractible.item:almostcoprod}), $|\Mod(T)|$ is contractible and in particular $\Gal_L(T)$ is trivial.
 \end{eg}
 
 \begin{eg}[o-minimal expansions of real closed fields]\label{eg:rcf}
 Let $T$ be the theory of some o-minimal expansion of real-closed fields. Recall that $T$ has $\emptyset$-definable Skolem functions \cite{Lou}. So the definable closure of the empty set is a model, which is there fore an initial object in $\Mod(T)$ (cf. \ref{rmk:skolem}). Hence by Proposition \ref{prop:contractible}(\ref{prop:contractible.item:initial}), $|\Mod(T)|$ is contractible and $\Gal_L(T)$ is trivial.
 \end{eg}
 
 \begin{eg}[Algebraically Closed Fields]\label{eg:acf}
 Let $T$ be the theory of algebraically-closed fields of characteristic $p$ (where $p$ is prime or 0). Let $k$ be the prime field, $\bar k$ its algebraic closure, and $G = \Gal(\bar k / k)$. Arbitrarily choose an embedding $\bar k \to M$ for each $M \in \Mod(T)$. Each field embedding induces an isomorphism of copies of $\bar k$, and this induces a functor $F: \Mod(T) \to \bbB G$, which is a Grothendieck opfibration (recall from Recollection \ref{rec:fibration} that a lift of an isomorphism is (co)cartesian iff it is an isomorphism). The fiber is equivalent to the the category of algebraically-closed $\bar k$-algebras, which has an initial object given by the identity $\bar k \to \bar k$, and so has contractible classifying space by Proposition \ref{prop:contractible}(\ref{prop:contractible.item:initial}). By Quillen's Theorem A (Theorem~\ref{thm:quillen-a}), $F$ is a homotopy equivalence. Hence $|\Mod(T)| \simeq |\bbB G| \cong BG^\delta$ (where $G^\delta = G$, considered as a discrete group). In particular, $\Gal_L(T) \cong  G$.
 \end{eg}
 
 \begin{eg}[Random Graphs]\label{eg:ran}
 Let $T$ be the theory of a random graph. Recall that $T$ may be axiomatized in the graph language by saying that for every two finite sets $S_1,S_2$ of vertices, there is a point $v$ which is connected by an edge to all the points of $S_1$ and none of the points of $S_2$. In this language, $T$ admits quantifier elimination and so is model-complete. We will show that $|\Mod(T)|$ is contractible; in particular, $\Gal_L(T)$ is trivial.
 
 First note that the category $\calG$ of (loop-free, undirected, simple) graphs and embeddings between them has a contractible classifying space; for instance the empty graph is an initial object. (In general, for any signature $\Sigma$, the category of $\Sigma$-structures and embeddings splits as the disjoint union of several components, each of which has an initial object, so the classifying space is discrete up to homotopy.) We will now construct functors and natural transformations between $\calG$ and $\Mod(T)$ satisfying the hypotheses of Fact \ref{fact:catequiv}(\ref{fact:catequiv.item:minus}), showing that $|\Mod(T)| \simeq |\calG|$ is contractible.
 
 There is a functor $F: \calG \to \Mod(T)$ which sends a graph $\Gamma$ to the graph $F\Gamma = \cup_{n = 0}^\infty \Gamma_n$ built up as follows. We take $\Gamma_0 = \Gamma$. Given $\Gamma_n$, $\Gamma_{n+1}$ consists of $\Gamma_n$ along with, for each finite set $S$ of vertices in $\Gamma_n$, a specified vertex $v(S)$, which has edges connecting it to each element of $S$, but no other edges in $\Gamma_{n+1}$. It is clear that $F \Gamma$ is a model of $T$, i.e. a random graph. A graph embedding $f: \Gamma \to \Gamma'$ is extended to an embedding $Ff: F\Gamma \to F \Gamma'$ by inductively defining $Ff(v(S)) = v(f(S))$. By model completeness, $Ff$ is elementary, and $F$ is clearly functorial. There is also a forgetful functor $U: \Mod(T) \to \calG$. Now, $F$ and $U$ are not adjoint, but there are natural transformations $\id_\calG \Rightarrow UF$ , given by embedding $\Gamma$ into $F\Gamma$ in the natural way, and $\id_{\Mod(T)} \Rightarrow FU$ ,also given by embedding $\Gamma$ into $F\Gamma$ in the natural way; this embedding is elementary by model completeness.
 \end{eg}
 
 \begin{rmk}
 The method of proof in Example \ref{eg:ran} is very suggestive. As the referee observed, it seems likely that something similar can be done for more general Fra\"i ss\'e classes. We do not pursue this possibility in this paper, although we re-use the same method in the next example.
 \end{rmk}
 
\begin{eg}[Triangle-free Random Graphs]\label{eg:tri-free-ran}
Let $T$ be the theory of triangle-free random graph as in \cite{Henson}. Recall that $T$ may be axiomatized by saying that the graph is triangle free and moreover for every finite sets $S_1,S_2$ of vertices where the induced subgraph on $S_1$ is discrete, there is a vertex connected by an edge to each point of $S_1$ which is not connected to any point of $S_2$. As in the example of random graphs, $T$ admits quantifier elimination and so is model complete.

Let $\calG$ be the category of triangle-free graphs and strong embeddings between them as morphisms. This category $\calG$ has an initial object hence the classifying space is contractible. We proceed as the previous example, \ref{eg:ran}, to construct a functor $F:\calG \to \Mod(T)$. Let $\Gamma$ be an arbitrary triangle-free graph. Take $\Gamma_0=\Gamma$. Assume $\Gamma_n$ has been constructed, $\Gamma_{n+1}$ consists of $\Gamma_n$ along with, for each finite set $S\subseteq \Gamma_n$ such that the induced subgraph on $S$ is discrete, a specified vertex $v(S)$ with the property that for any vertex $w$, there is an edge between $w$ and $v(S)$ iff $w\in S$. It is routine to check that the final construction yields a triangle-free random graph. One can similarly define extensions of embeddings as in \ref{eg:ran}. Let $U$ be the forgetful functor from $\Mod(T)$ to $\calG$. The same argument as in \ref{eg:ran} shows that $U,F$ induce homotopy equivalences, hence $|\Mod(T)|$ is contractible.
\end{eg}

\begin{rmk}
In Example \ref{eg:tri-free-ran}, $\Mod(T)$ fails to have 3-amalgamation. As mentioned in the Introduction, this illustrates that the failure of higher amalgamation does not imply that $|\Mod(T)|$ has interesting higher homotopy groups.
\end{rmk}
 
  \begin{eg}[Algebraically Closed Valued Fields]
 Let $T$ be $\mathrm{ACVF}_{q,p}$, where $q,p$ stands for the characteristic of the valued field and residue field respectively. There is a forgetful functor $U: \Mod(T) \to \Mod(T_1) \times \Mod(T_2)$ where $T_1$ is the theory of algebraically closed fields of characteristic $p$ and $T_2$ is the theory of divisible ordered abelian groups, given by taking the residue field and the valuation group, respectively. Note that ACVF has quantifier elimination, hence there is a functor in the other direction $F: \Mod(T_1) \times \Mod(T_2) \to \Mod(T)$ which takes Hahn series, when $T$ is the equicharacteristic case and in the mixed characteristic case, take the $p$-adic Mal'cev series as in \cite{Poonen}. Again, these functors are not adjoint, but $UF$ is the identity, and by \cite[Theorems 5 and 6]{kap1942} and
 \cite[Theorem 2]{Poonen}, there is a embedding $M \to FU(M)$ for $M \in \Mod(T)$, which corresponds to taking the maximal immediate completion. Moreover, the embedding is natural as the algebraically closed valued fields satisfies Kaplansky's ``hypothesis A" \cite{kap1942}, hence the maximal immediate completion is natural. So by Fact \ref{fact:catequiv}(\ref{fact:catequiv.item:minus}), $|\Mod(T)| \simeq |\Mod(T_1)| \times |\Mod(T_2)|$, so that $\Gal_L(T) \cong \Gal_L(T_1) \times \Gal_L(T_2)$. We have seen in Example \ref{eg:doag} that $|\Mod(T_2)|$ is trivial and in Example \ref{eg:acf} that $|\Mod(T_1)| \simeq BG^\delta$, where $G$ is the Galois group of the prime model. So $|\Mod(T)| \simeq BG^\delta$ and $\Gal_L(T) \cong G$.
 \end{eg}
 
 \begin{eg}[$G$-torsors]\label{eg:torsor}

 Let $G$ be a compact Lie group and let $T_0$ be an expansion of the theory of real-closed fields by finitely many real-analytic functions on bounded rectangles such that $G$ can be defined over $T_0$, note that $T_0$ is o-minimal and has definable Skolem functions, hence admits an initial model. Let $T$ denote the theory of a structure $(R,X)$ where $R$ is a model of $T_0$ and $X$ is a set equipped with a $G(R)$-action such that $X$ is a $G(R)$-torsor. Ziegler \cite{zie} showed that $\Gal_L(T) = G(\reals)$. 
 We will recover this result (at the level of discrete groups, of course). However, in this case, we do not identify the higher homotopy groups of $\Mod(T)$.
 
 For $(R,X) \in \Mod(T)$. We define and equivalence relation $\sim$ on $G(R)$ by $x\sim y$ if $x,y$ differs by an element in the infinitesimal subgroup of $G$. Then by the fact that $G$ is definably compact, $(G(R)/\sim)^\wedge \cong G(\reals)$, where $(-)^\wedge$ denotes the completion. Similarly, let $X(\reals)  = (X/ \sim)^\wedge$ where $x \sim y$ if there is an infinitesimal $g \in G(R)$ such that $g x = y$, and $(-)^\wedge$ denotes taking the completion. This construction yields a functor $F: \Mod(T) \to G(\reals)\Tor$, where $G(\reals)\Tor$ is the groupoid of torsors over $G(\reals)$; note that there is an equivalence of groupoids $G(\reals)\Tor \simeq \bbB G(\reals)^\delta$. The functor $F$ has a section sending $X \mapsto (\reals,X)$. Thus the induced map $F_\ast: \Gal_L(T) = \pi_1(|\Mod(T)|) \to \pi_1(G(\reals)\Tor) = G(\reals)$ is surjective.
 
 Now we show that $F_\ast : \pi_1(|\Mod(T)|, (R,X(R))) \to G(\reals)$ is injective, where we have chosen the base-point $(R,X(R))$ such that $R$ is the initial model of $T_0$. Consider an arbitrary group element $\gamma \in \pi_1(|\Mod(T)|,(R,X(R))$. Then by Theorem \ref{thm:mainthm}, we have $\Gal_L(T,\bbU)\cong \pi_1(\Mod(T),\bbU)\cong \pi_1(\Mod(T),(R,X))$, we may represent $\gamma = [j]\inv [i]$ where $i: (R,X(R)) \to \bbU$ is a standard embedding into a monster model $\bbU$, and $j: (R,X(R)) \to \bbU$ differs by an automorphism. Let $x,y$ be the images of a named base-point of the torsor parts of $i(R,X(R))$, $j(R,X(R))$ respectively. Assume that $\gamma\in \ker F_\ast$. Then $x$ and $y$ differ by an infinitesimal element of $G(\bbU)$. It is easy to see that for two elements in the torsor sort, their type is determined by the type of the group element they differ by. Thus we may find $z$ in the torsor part of $\bbU$ such that $\tp(x,z) = \tp(y,z)$, and so there is an automorphism $\alpha$ of $\bbU$ taking $y$ to $x$ and fixing $z$. To find such $z$, it suffices to show the following: given $\epsilon\in G$ an infinitesimal, we can find $g\in G$ such that $\tp(g\epsilon)=\tp(g)$. Take $g$ to be a realization of a $f$-generic type on $G$, note that $\epsilon \in G^{00}$, hence by \cite[Lemma 8.18]{pierre-nip-book}, we have $\tp(g\epsilon)=\tp(g)$ and we can simply take $\epsilon y=x$ and $gx=z$. Let $k: (R,X(R)) \to \bbU$ be an embedding with $z$ in its image. Then $\gamma = [j]\inv [i] = [\alpha i]\inv [i] = [i]\inv [\alpha]\inv [k] [k] \inv [i] = [i]\inv [\alpha \inv k] [k]\inv [i] =  [i]\inv [k] [k]\inv [i] = 1$. So $F_\ast$ is injective as well as surjective, and so an isomorphism, and $\Gal_L(T) \cong G(\reals)$. 
 \end{eg}
 
 
\section{Higher Homotopy groups}\label{sec:higher-hty}

We do not know in general the answer to the following

\begin{question}\label{quest:hom-type}
Which homotopy types can be realized as $|\Mod(T)|$ for some complete theory $T$?
\end{question}

If the class of such homotopy types is relatively ``diverse" (in particular, if it includes some non-aspherical spaces), it would introduce new invariants to model theory, such as the higher homotopy groups, homology and cohomology groups of $|\Mod(T)|$, and it would be worth investigating their model theoretic meaning.

We temper our enthusiasm for this possibility with the reality that we don't even know an example of a theory $T$ for which $|\Mod(T)|$ is not aspherical. In Subsection \ref{subsec:higher-hty-2} we consider more restrictive categorical properties that we obtain in $\Mod(T)$ for certain theories $T$ which imply that $|\Mod(T)|$ is aspherical. In Subsection \ref{subsec:higher-hty-1}, we formulate Question \ref{quest:aec-hty}, an easier, purely categorical analog of Question \ref{quest:hom-type}, and find some partial answers (Observation \ref{obs:nmm}, Observation \ref{obs:ap}, Theorem \ref{thm:aechty}). 

\subsection{Realizing arbitrary homotopy types}\label{subsec:higher-hty-1}

In our work thus far we have used the following properties of the category $\Mod(T)$:
\begin{rmk}\label{rmk:misc-props}
The following are several properties that a category $\calC$ may satisfy (and which are satisfied in the case $\calC = \Mod(T)$):
\begin{enumerate}
\item\label{item:jep} $\calC$ has the joint embedding property (JEP) (cf. Definition \ref{def:JEP-AP}).
\item\label{item:ap} $\calC$ has the amalgamation property (AP) (cf. Definition \ref{def:JEP-AP}).
\item\label{item:nmm} $\calC$ has no maximal models. That is, for every object $X \in \calC$, there is an object $Y$ in $\calC$ and a morphism $X \to Y$ which is not an isomorphism.
\item\label{item:mono} Every morphism is a monomorphism (cf. Recollection \ref{rec:mono}).
\item\label{item:acccat} $\calC$ is an accessible category (cf. Definition \ref{def:acc}).
\item\label{item:filco} $\calC$ has filtered colimits.
\item\label{item:aec} $\calC$ is equivalent to an Abstract Elementary Class (AEC) in the sense of \cite{shelah2009classification}.
\item\label{item:monster} $\calC$ has a $\kappa$-universal and $\kappa$-homogeneous object for some $\kappa$ (in the sense of Definition \ref{def:univ-homog-etc} and Remark~\ref{rmk:univ-homog} ).
\end{enumerate}
\end{rmk}

Note that JEP (\ref{item:jep}) implies that $\calC$ is connected. As a sort of converse, if AP (\ref{item:ap}) holds, then each connected component of $\calC$ satisfies JEP (\ref{item:jep}). It may seem that (\ref{item:monster}) is the most powerful condition, but any category satisfying (\ref{item:jep}),(\ref{item:ap}),(\ref{item:acccat}), and (\ref{item:filco}) actually satisfies (\ref{item:monster}) as well, at least under mild set-theoretical hypotheses \cite[Statement 2.2]{lieberman-rosicky}. Note that (\ref{item:aec}) is equivalent to the conjunction of (\ref{item:mono}),(\ref{item:acccat}),(\ref{item:filco}) and the existence of a faithful, nearly full, full-on-isomorphisms, monomorphism-preserving and filtered-colimit-preserving functor $\calC \to \calD$ for some finitely accessible category $\calD$ \cite[Corollary 5.7]{beke-rosicky}.\footnote{The relationship to the definition of \cite{shelah2009classification} is that the objects of $\calC$ are the models of the AEC and the morphisms of $\calC$ are composites of isomorphisms and inclusions of ``abstract elementary substructures"; the category $\calD$ is essentially the category of structures for the language and all embeddings, and the functor $\calC \to \calD$ is the forgetful functor.} We have not defined what a nearly full functor is, but the following will suffice for our purposes.
\begin{fact}\label{fact:fin-acc-AEC}
Every finitely accessible category of monomorphisms $\calC$ is equivalent to an AEC.
\end{fact}
\begin{proof}
By \cite[Corollary 5.7]{beke-rosicky}, it suffices to define a full with respect to isomorphisms and nearly full embedding  from $\calC$ to some finitely accessible category $\calD$ which preserves filtered colimits and monomorphisms. The identity functor $\id_\calC$ is such an embedding.
\end{proof}

Our categorical analog of Question \ref{quest:hom-type} thus takes the following form:

\begin{question}\label{quest:aec-hty}
Let $X$ be a CW complex. For which subsets of the properties (1-8) of Definition \ref{rmk:misc-props} can we find categories $\calC$ with $|\calC| \simeq X$?
\end{question}

We obtain a partial solution:

\begin{thm}\label{thm:aechty}
Let $\kappa$ be a small regular cardinal. Then every CW complex $X$ is homotopy equivalent to the classifying space of a finitely accessible category $\calE$ of monomorphisms with amalgamation for $\kappa$-presentable objects. In particular, $\calE$ is an AEC with amalgamation for $\kappa$-presentable objects.
\end{thm}
The category $\calE$ is different from $\Mod(T)$ in general since $\Mod(T)$ is rarely finitely accessible when $T$ has an infinite model. However, $\Mod(\mathrm{ACF})$ is indeed finitely accessible.

We provide below a construction (Construction \ref{const:mainconst}) validating Theorem \ref{thm:aechty}. Thus we attain, for an arbitrary homotopy type, property (\ref{item:aec}), and \emph{a fortiori} properties (\ref{item:mono}), (\ref{item:acccat}), (\ref{item:filco}), and (at least under set-theoretical hypotheses) (\ref{item:monster}) of Remark (\ref{rmk:misc-props}), but only a weakened form of properties (\ref{item:jep}) and (\ref{item:ap}). We do not know whether property (\ref{item:nmm}) will always hold in Construction \ref{const:mainconst}. It should be noted that in the present construction, we will have $\calE = \Ind(\calD)$ for a category $\calD$ with $\#\calD \geq \kappa$.

In a separate paper \cite{campion-ye}, the first and third author have provided an alternate construction yielding a category $\calE$ which is canonically isomorphic to an AEC with amalgamation for all objects and no maximal models, thus attaining all 8 properties of Remark \ref{rmk:misc-props} for an arbitrary homotopy type (with the caveat, of course, that (\ref{item:jep}) is attained only for connected homotopy types).


Before embarking on the proof of Theorem \ref{thm:aechty}, we make some easier observations.

\begin{obs}\label{obs:nmm}
Every homotopy type is realized as $|\calC|$ where $\calC$ is an AEC with no maximal models. In particular, $\calC$ satisfies (\ref{item:nmm}),(\ref{item:mono}),(\ref{item:acccat}),(\ref{item:filco}),(\ref{item:aec}) above.
For by subdivision, every homotopy type $X$ is the classifying space of a (small) poset $P$ (Fact \ref{fact:poset}). Then $\Ind(P)$ is again a small poset and so $\Ind(P)\times \Set^{\geq \omega}_\inj$ is again a category of monomorphisms (by Lemma \ref{lem:ind-mono}) and no maximal models, where $\Set^{\geq \omega}_\inj$ is the category of small infinite sets and injections. $\Ind(P)\times \Set^{\geq \omega}_\inj$ is finitely accessible since a finite product of accessible categories is again finitely accessible \cite[Proposition 6.1.12]{kashiwara-schapira}. Hence (3),(4),(5),(6),(7) in Remark \ref{rmk:misc-props} are satisfied. We have $|\Ind(P) \times \Set^{\geq \omega}_\inj| \simeq |\Ind(P)| \times |\Set^{\geq \omega}_\inj| \simeq |P| \times |\Delta^0| \simeq X$.
\end{obs}

\begin{obs}\label{obs:ap}
It is not difficult to formally adjoin amalgamation to a category $\calC$: for each amalgamation problem $B \leftarrow A \to C$, adjoin an object $B \ast_A C$ with $\Hom(X,B \ast_A C) = \Hom(X,B) \cup_{\Hom(X,A)}\Hom(X,C)$ for $X \in \calC$ and no other nonidentity morphisms besides those in $\calC$. Apply this construction iteratively. This construction does not change the homotopy type if $\calC$ is a category of monomorphisms -- the argument is similar to the one found in the proof of Theorem \ref{thm:aechty} below. Thus any homotopy type $X$ is realized by a category $\calC$ satisfying (\ref{item:ap}),(\ref{item:nmm}), (\ref{item:mono}) of Remark \ref{rmk:misc-props}, and if $X$ is connected, then $\calC$ also satisfies (\ref{item:jep}).

Unfortunately, the $\Ind$ construction does not preserve the amalgamation property \cite{kruckman-MO}, and so although $\Ind(\calC)$ is an AEC, it may fail to have amalgamation.
\end{obs}


We now introduce a construction which will validate Theorem \ref{thm:aechty}.

\begin{construction}\label{const:mainconst}
Let $\calC$ be a small category and $\kappa$ be a small regular cardinal. Let $K$ be the poset $2 \times (\kappa+1)$. By induction on $(i,\alpha) \in K$, we construct a diagram of small categories $(\calC^k)_{k \in K}$ indexed by $K$; we will write $\calD^\alpha = \calC^{(0,\alpha)}$ and $\calC^\alpha = \calC^{(1,\alpha)}$. The construction is as follows:

\noindent\textbf{Initial Step:} $\calD^0 = \calC$ 

\noindent\textbf{First Successor Step:} In this step of the construction, we use the facts stated in Recollection \ref{def:ind} about $\Ind$ constructions.
Given $\calD^\alpha$, we define $\calC^\alpha = \Ind^\kappa(\calD^\alpha)$, and $\calD^\alpha \to \calC^\alpha$ the canonical inclusion $y_{\calD^\alpha}$. For $\beta \leq \alpha$, the functor $\calC^\beta \to \calC^\alpha$ is induced functorially from $\calD^\beta \to \calD^\alpha$. 

\noindent\textbf{Second Successor Step:} Given $\calC^\alpha$, we define $\calD^{\alpha+1}$ as follows.

An \textbf{Object} of $\calD^{\alpha+1}$ is either an object of $\calD^\alpha$, or else consists of a span $B \leftarrow A \to C$ in $\calC^\alpha$. In the latter case, the object will be denoted $B \ast_A C$.

\textbf{Morphisms} of $\calD^{\alpha+1}$ are as follows:
    \[\calD^{\alpha+1}(X,Y) = 
        \begin{cases}
            \calD^\alpha(X,Y) & X,Y \in \calD^\alpha \\
            \{\id_X\} & X = Y = B \ast_A C \\
            \emptyset & X = B \ast_A C, X \neq Y \\
            \calC^\alpha(X,B) \cup_{\calC^\alpha(X,A)} \calC^\alpha(X,C) & X \in \calD^\alpha, Y = B \ast_A C
        \end{cases}
    \]

\textbf{Composition} in $\calD^{\alpha+1}$ is defined in the obvious way.
    
\noindent\textbf{Limit Step:} If $\alpha$ is a limit ordinal, we define $\calD^\alpha = \cup_{\beta  < \alpha} \calD^\beta$ (recall that $\calC^\alpha = \Ind^\kappa(\calD^\alpha)$ even when $\alpha$ is a limit).

This completes Construction \ref{const:mainconst}.
\end{construction}

\begin{notation}
We fix some notation for referring to objects and morphisms of the categories $(\calC^k)_{k \in K}$.
\begin{enumerate}
    \item Let $\alpha <\kappa$, note that if $X,Y \in \calC^\alpha$, then $X = \varinjlim_p X_p$ and $Y = \varinjlim_q Y_q$ for some filtered diagrams $(X_p)_{p \in P}, (Y_q)_{q \in Q}$ in $\calD^\alpha$. Recall from Recollection \ref{def:ind} that we have $\Hom(X,Y) = \varprojlim_{p \in P}( \varinjlim_{q \in Q} \Hom(X_p,Y_q))$, so that a morphism $f: X \to Y$ consists of a $P$-tuple of equivalence classes of diagrams $(f_p)_{p \in P}$, where $f_p : X_p \to Y_q$, for some $q$ (dependent on $p$).
    \item Let $\alpha < \kappa$. If $X \in \calD^\alpha$ and $Y \in \calD^{\alpha+1}\setminus\calD^\alpha$, then $Y = B \ast_A C$ for some $B \overset{f}{\leftarrow} A \xrightarrow{g} C \in \calC^\alpha$. If $h \in \calD^{\alpha+1}(X,Y) = \calC^\alpha(X,B) \cup_{\calC^\alpha(X,A)}\calC^\alpha(X,C)$, then we may write $h = (B,h')$ for some $h' \in \calC^{\alpha}(X,B)$, or $h = (C,h')$ for some $h' \in \calC^{\alpha}(X,C)$, or $h = (A,h')$ for some $h' \in \calC^\alpha(X,A)$. These expressions are unique up to the equivalence relation generated by the fact that if $k: Z \to A$ is a morphism in $\calC^\alpha$, then $(A,k) = (B,fk) = (C,gk): Z \to B \ast_A C$.
    \item We denote the various structure functors between the $\calC^k$ as $\Gamma: \calC^\beta \to \calC^\alpha$, $\Delta: \calD^\beta \to \calD^\alpha$, and $y: \calD^\alpha \to \calC^\alpha$.
\end{enumerate}
\end{notation}

\begin{lem}\label{lem:aec}
Let $\kappa$ be a small regular cardinal and $\calC$ a small category of monomorphisms. Consider the categories $(\calC^k)_{k \in K}$ of Construction \ref{const:mainconst}.
\begin{enumerate}
    \item\label{lem:aec.item:colim} $\calC^{\beta}$ has $\kappa$-small filtered colimits, preserved by the functor $\Gamma : \calC^{\beta} \to \calC^{\alpha}$ for $\beta \leq \alpha$. 
    \item\label{lem:aec.item:pres} If $X \in \calD^\alpha$, then $\calC^\alpha(X,-)$ commutes with colimits of $\kappa$-small directed systems in $\calC^{\alpha}$.
    \item\label{lem:aec.item:gen} Every object $X \in \calC^{\alpha}$ is a colimit of a $\kappa$-small directed system of objects of $\calD^{\alpha}$.
    \item\label{lem:aec.item:monic} Every morphism of $\calC^{(i,\alpha)}$ is monic.
    \item\label{lem:aec.item:heq} The inclusion $\calC^{k} \to \calC^{l}$ induces a homotopy equivalence $|\calC^{k}|\to|\calC^{l}|$ for all $k \leq l \in K$.
    \item\label{lem:aec.item:AP} Any span in $\calC^{\alpha}$ has an amalgam in $\calC^{\alpha+1}$.
\end{enumerate}
In particular, $\calC^\kappa$ is a category of monomorphisms with amalgamation and colimits of $\kappa$-small directed systems and $\calC \to \calC^\kappa$ is a homotopy equivalence. Moreover every object of $\calC^\kappa$ is a colimit of a $\kappa$-small directed system of objects of $\calD^\kappa$, and for every $X \in \calD^\kappa$, $\calC^\kappa(X,-)$ commutes with colimits of $\kappa$-small directed systems.
\end{lem}
\begin{proof}
(\ref{lem:aec.item:colim}), (\ref{lem:aec.item:pres}), and (\ref{lem:aec.item:gen}) follow from the remarks in Recollection \ref{def:ind-res}.


We show (\ref{lem:aec.item:monic}) by induction on $(i,\alpha)$. First, $\calC$ is a category of monomorphisms by hypothesis. Next, if $\calD^\alpha$ is a category of monomorphisms, then so is $\calC^\alpha$ by Lemma \ref{lem:ind-mono}. The case of $\calD^\alpha$ where $\alpha$ is a limit ordinal is trivial. It remains to show that if $\calC^\alpha$ is a category of monomorphisms, then so is $\calD^{\alpha+1}$. This necessitates an analysis of morphisms into $B \ast_A C$, the amalgam of a span $B \overset{f}{\leftarrow} A \xrightarrow g C$ in $\calC^\alpha$. By induction $f,g$ are monic in $\calC^\alpha$. It follows that for $D \in \calD^\alpha \subseteq \calD^{\alpha+1}$, the only identifications between morphisms in $\calD^{\alpha+1}(D, B\ast_A C)$ are the identifications $(B,fh) = (A,h) = (C,gh)$ for $h: D \to A$ in $\calC^\alpha$ i.e. these identifications already form an equivalence relation rather than merely generating an equivalence relation (see \ref{eg:colim}). Therefore if $(B,k) = (B,l)$ for $k,l : D \to B$ in $\calC^\alpha$, then either we may immediately conclude that $k=l$, or else we have $k = fk'$ and $l = fl'$ for $k',l': D \to A$ in $\calC^\alpha$, and we have $(A,k') = (A,l')$, from which we may conclude that $k' = l'$. Either way, $(B,l) = (B,k)$ implies that $l=k$. Now we are ready to show that $\calD^{\alpha+1}$ is a category of monomorphisms. The only nontrivial case to check is a morphism $D \to B \ast_A C$ for $D \in \calD^\alpha$; without loss of generality this morphism is of the form $(B,h)$ for some $h: D \to B$ in $\calC^\alpha$. To see that this morphism is monic, consider two morphisms $n,m: E \rightrightarrows D$ in $\calD^{\alpha+1}$; necessarily we have $E \in \calD^{\alpha}$. Suppose that $(B,hn) = (B,hm)$. Then we have seen that we may conclude that $hn = hm$ in $\calC^\alpha$. By induction, $h$ is a monomorphism in $\calC^\alpha$, and so we have $n=m$ in $\calC^\alpha$. Because the inclusion $y: \calD^\alpha \to \calC^\alpha$ is faithful, this is to say that $n=m$ when viewed as morphisms in $\calD^\alpha$. So by definition, $n=m$ when viewed as morphisms in $\calD^{\alpha+1}$, and so $(B,h)$ is a monomorphism as desired.

For (\ref{lem:aec.item:heq}), we proceed by induction on $l$. First suppose that $l=(1,\alpha)$, so $\calC^l = \calC^\alpha$. Consider first the inclusion $\calD^\alpha \to \calC^\alpha$. As in Remark \ref{rmk:ind-eq}, $\calD^\alpha \downarrow X$ is filtered for each $X \in \calC^\alpha$ and in particular contractible, so by Quillen's Theorem A (Theorem ~\ref{thm:quillen-a}), the inclusion induces a homotopy equivalence. For $\beta < \alpha$, the functor $\calD^\beta \to \calD^\alpha$ induces a homotopy equivalence by induction, and thus by composition the functor $\calD^\beta \to \calC^\alpha$, which factors as $\calD^\beta \to \calD^\alpha \to \calC^\alpha$, also induces a homotopy equivalence. Finally, for $\beta < \alpha$ the functor $\calC^\beta \to \calC^\alpha$ becomes the homotopy equivalence $\calD^\beta \to \calC^\alpha$ after precomposition with the homotopy equivalence $\calD^\beta \to \calC^\beta$, and so is also a homotopy equivalence since homotopy equivalences are isomorphisms in the homotopy category (~Definition \ref{def:homotopy}).

Next suppose that $l = (0,\alpha)$ for $\alpha$ a limit ordinal, so that $\calC^l = \calD^\alpha$. Then for any $\beta < \alpha$, the inclusion $\calD^\beta \to \calD^\alpha$ is a filtered colimit of homotopy equivalences, and hence induces a homotopy equivalence (Fact \ref{fact:fil-whe}). Finally, suppose that $l = (0,\alpha+1)$ so that $\calC^l = \calD^{\alpha+1}$; it will suffice to show that the inclusion $\calD^\alpha \to \calD^{\alpha+1}$ induces a homotopy equivalence. By Quillen's Theorem A( Theorem~\ref{thm:quillen-a}) it suffices to show that for each $X \in \calD^{\alpha+1}$, the slice category $\calD^\alpha \downarrow X$ is contractible. If $X \in \calD^\alpha$, (or more generally if $X \in \calC^\alpha$) then we have just seen that $\calD^\alpha \downarrow X$ is contractible. If $X = B \ast_A C$ for $A,B,C \in \calC^\alpha$, then by the considerations given in the proof of (\ref{lem:aec.item:monic}) above, we have $(\calD^\alpha\downarrow X)_n = (\calD^\alpha\downarrow B)_n \cup_{(\calD^\alpha\downarrow A)_n} (\calD^\alpha\downarrow C)_n$ for each $n \in \nats$, where $\calA_n$ denotes the set $n$-chains of composable morphisms of a category $\calA$ as in Definition \ref{def:classifying-space}. Moreover, the pushout is along injective maps by (\ref{lem:aec.item:monic}). Thus we have a pushout of simplicial sets along injections: $\nerve (\calD^\alpha \downarrow X) = \nerve (\calD^\alpha \downarrow B) \cup_{\nerve (\calD^\alpha \downarrow A)} \nerve (\calD^\alpha \downarrow C)$ (here $\nerve$ denotes the nerve functor -- cf. Remark \ref{rmk:nerve-realization}). Because geometric realization preserves pushouts (Fact \ref{fact:geom-colim}) and takes injections to injective cellular maps, we obtain a pushout of spaces along cellular maps: $|\calD^\alpha \downarrow X| = |\calD^\alpha \downarrow B| \cup_{|\calD^\alpha \downarrow A|}|\calD^\alpha \downarrow C|$. We have seen that the spaces $|\calD^\alpha \downarrow A|,|\calD^\alpha \downarrow B|,|\calD^\alpha \downarrow C|$ are contractible, so by Fact \ref{fact:po-contr}, it follows that $|\calD^\alpha \downarrow X|$ is also contractible.

(\ref{lem:aec.item:AP}): Note that $\Gamma y = y \Delta$. Let $B \overset f \leftarrow A \overset g \to C$ be a span in $\calC^\alpha$; we would like to show that $\Gamma(B \overset f \leftarrow A \overset g \to C)$ has an amalgam in $\calC^{\alpha+1}$. We claim that $y(B \ast_A C)$ is such an amalgam. This is not immediately obvious, since there is no functor $\calC^\alpha \to \calD^{\alpha+1}$, so we check it carefully. First, we define morphisms $\Gamma(y(B)) \xrightarrow j y(B \ast_A C)$, $\Gamma(y(C)) \xrightarrow k y(B \ast_A C)$ in $\calC^{\alpha+1}$. Write $A,B,C$ as $\kappa$-small filtered colimits of objects of $\calD^\alpha$: $A = \varinjlim_p y(A_p)$, $B = \varinjlim_q y(B_q)$, $C = \varinjlim_r y(C_r)$, and write $i_p: y(A_p) \to A$, $j_q: y(B_q) \to B$, $k_r: y(C_r) \to C$ for the canonical inclusions. Because $\Gamma: \calC^\alpha \to \calC^{\alpha+1}$ preserves $\kappa$-small filtered colimits by (\ref{lem:aec.item:colim}), these colimit expressions are equally valid in $\calC^{\alpha+1}$, e.g. $\Gamma(B) = \varinjlim_q y(\Delta(B_q))= \varinjlim_q \Gamma(y(B_q)) $. Consider the morphisms $(B,j_q): \Delta(B_q) \to B \ast_A C$ for each $q$. We claim that these form a cocone, i.e. that $(B,j_q) \Delta(B_\theta) = (B,j_{q'})$ for each morphism $\theta: q' \to q$. This follows from the fact that $j_q B_\theta = j_{q'}$ and the definition of composition in $\calD^{\alpha+1}$. Thus, we obtain a morphism $j := (y(B,j_q))_q : \Gamma(y(B)) = y(\Delta(B)) \to y(B \ast_A C)$. Similarly, there is a morphism $k:= (y(C,k_r))_r: \Gamma(y(C)) = y(\Delta(C)) \to y(B \ast_A C)$. Now we claim that $j,k$ are an amalgam of $\Gamma(f),\Gamma(g)$, i.e. that $j\Gamma(f) = k\Gamma(g)$. By the universal property of the colimit $\Gamma(A) = \varinjlim_p \Gamma(y(A_p)) = \varinjlim_p y(\Delta(A_p))$, it suffices to check that $j\Gamma(f)\Gamma(i_p) = k\Gamma(g) \Gamma(i_p)$ for each $p$. Now $j \Gamma(f)\Gamma(i_p) = j \Gamma(fi_p)$. Because the colimit $B = \varinjlim_q y(B_q)$ is filtered and $\kappa$-small and $A_p \in \calD^\alpha$, by (\ref{lem:aec.item:pres}) we have $fi_p = j_q y(\tilde f)$ for some $\tilde f: A_p \to B_q$ in $\calD^\alpha$. So $j \Gamma(fi_p) = j \Gamma(j_q y(\tilde f)) = j \Gamma(j_q) \Gamma(y(\tilde f))$. Now by definition of $j$, we have $j \Gamma(j_q) = y(B,j_q)$. Thus $j \Gamma(j_q) \Gamma(y(\tilde f)) = y((B,j_q))y(\Delta(\tilde f)) = y((B,j_q) \Delta(\tilde f)) = y(B,j_q y(\tilde f)) = y(B,fi_p) = y(A,i_p)$. Similarly, we have $k\Gamma(g) \Gamma(i_p) = y(A,i_p)$, so the claim is verified.
\end{proof}

\begin{proof}[Proof of Theorem \ref{thm:aechty}]
By Fact \ref{fact:poset}, we may assume that $X$ is the classifying space of a small poset $P$. Let $\calC = P$ in Construction \ref{const:mainconst}, obtaining categories $\calC^\kappa, \calD^\kappa$. Set $\calE = \Ind_\kappa(\calC^\kappa) = \Ind_\kappa(\Ind^\kappa(\calD^\kappa)) = \Ind(\calD^\kappa)$, where the last equation follows from Fact \ref{fact:accprops-ind-ind}. So $\calE = \Ind(\calD^\kappa)$ is a finitely accessible category of monomorphisms, so $\calE$ is an AEC by Fact \ref{fact:fin-acc-AEC}. Finally, following the discussion in Recollection \ref{def:ind}, the $\kappa$-presentable objects of $\calE$ are the retracts of objects in $\Ind^\kappa(\calD^\kappa) = \Ind^\kappa(\cup_{\alpha < \kappa}\calD^\alpha) = \varinjlim_{\alpha < \kappa} \Ind^\kappa(\calD^\alpha) = \varinjlim_{\alpha < \kappa} \calC^\alpha$, but because this is a category of monomorphisms, there are no nontrivial retracts and so these are precisely the $\kappa$-presentable objects. Note that spans in $\calC^\alpha$ have amalgams in $\calC^{\alpha+1}$ by Lemma \ref{lem:aec} (\ref{lem:aec.item:AP}). So $\calC^\kappa=\Ind^\kappa(\calD^\kappa)$ has amalgamation. That is, $\calE$ has amalgamation for $\kappa$-presentable objects. Moreover, by Remark \ref{rmk:ind-eq}, $\calE$ has the same homotopy type as $\calC^\kappa$, which in turn by Lemma \ref{lem:aec}(\ref{lem:aec.item:heq}) has the same homotopy type as $X$.
\end{proof}

\subsection{Criteria for asphericity}\label{subsec:higher-hty-2}
It is worth noting that there are natural conditions on a category $\calC$ implying that $|\calC|$ is aspherical.\\
One useful criterion comes from the following theorem of Par\'e:

\begin{thm}[{\cite[Theorem 2]{pare}}]\label{thm:pare}
A category has pullbacks if and only if it has all finite simply-connected limits.

Dually, a category has pushouts if and only if it has all finite simply-connected colimits.
\end{thm}

Here, a \emph{simply-connected limit} is a limit indexed by a diagram $F: \calD \to \calC$ where $|\calD|$ is simply-connected, and dually for simply-connected colimits.

\begin{cor}
If $\calC$ has pullbacks or pushouts, then $|\calC|$ is aspherical.
\end{cor}
In the proof, we freely use the tools of simplicial homotopy theory. For an introduction, see \cite{riehl}, and for a textbook account, see \cite{goerss_jardine}.
\begin{proof}
Let $n \geq 2$. An element of $\pi_n(|\calC|)$ is represented by a cellular map $f: X \to |\calC|$ from some simplicial complex $X$ weakly homotopy equivalent to $S^n$. After subdividing if necessary, we may assume that $X$ is the classifying space of a finite poset $X = |P|$, such that $f = |F|$ for some functor $F: P \to \calC$. Now, $|P| \simeq S^n$ is simply-connected, and $\calC$ has all simply connected limits (or colimits) by Theorem \ref{thm:pare} so $F$ has a cone or cocone in $\calC$, for example the (co)limiting one. This (co)cone allows us to extend $f$ to the inclusion of $X$ into a cone on $X$, which is contractible. So $f$ represents a trivial element in $\pi_n(|\calC|)$.
\end{proof}

For example, $\Mod(T)$ has pullbacks if the intersection of two elementary sub-models is always an elementary sub-model. We thank the referee for pointing out that the above is equivalent to saying that algebraically closed subsets are models. On the other hand, it seems that $\Mod(T)$ does not have pushouts when $T$ has an infinite model. We thank the referee for the following argument. Take $M \preceq \bbU$ to be a model and let $f\in\Aut(\bbU/M)$ be nontrivial. Then take $M_1 \succeq M$ such that $f(M_1)\neq M_1$ and let $M_2 = f(M_1)$. Now, $M_1\leftarrow M \to M_2$ admits both $\bbU$ (with
inclusions) and $M_2$ (with f and equality) as cocones, but there is no cocone embedding into both $M_2$ and $\bbU$.

Another criterion comes from Dwyer and Kan:

\begin{thm}[{\cite[Proposition 7.4]{dwyer-kan-II}}]
If $\calC$ admits a \emph{calculus of right fractions}, then $|\calC|$ is aspherical.
\end{thm}
Here $\calC$ is said to admit a \emph{calculus of right fractions} if the following two criteria hold:
\begin{enumerate}
\item For every cospan $x \to z \leftarrow y$, there is a cone.
\item If $fu = fv$, then there exists $g$ such that $ug = vg$.
\end{enumerate}
In the case of $\Mod(T)$, condition (2) is trivial because every morphism $f$ is a monomorphism. (1) is equivalent to the fact that any intersection of two elementary sub-models contains an elementary sub-model. And this is equivalent to the fact that $\acl(\emptyset)$ is a model.

\begin{cor}
Suppose $T$ has the property that $\acl(\emptyset)$ is a model, then $|\Mod(T)|$ is aspherical, so that $|\Mod(T)| \simeq B \Gal_L(T)$ and $\Gal_L(T)\cong \Aut(\acl(\emptyset))$.
\end{cor}

Unfortunately, when $T$ has an infinite model, $\Mod(T)$ never has the dual notion of a calculus of \emph{left} fractions because, though the dual of condition (1) is just the amalgamation property, the dual of condition (2) is never satisfied: since all morphisms of $\Mod(T)$ are monomorphisms, the dual of condition (2) would imply that $\Mod(T)$ was a poset. 

\printbibliography

\end{document}